\documentclass[11pt, a4paper]{article}

\usepackage{inputenc} 
\usepackage[english]{babel}
\usepackage{dsfont}
\usepackage{authblk}

\usepackage[totalheight=24 true cm, totalwidth=15 true cm]{geometry}
\usepackage{enumitem}

\setlist[itemize]{noitemsep}
\setlist[enumerate]{noitemsep}
\usepackage{bbding}
\usepackage{amsmath,amsthm,amssymb}
\usepackage{mathrsfs}
\usepackage{wasysym}
\usepackage{mathtools}
\usepackage{stmaryrd}
\usepackage{url}
\usepackage{wrapfig}
\usepackage{starfont}
\usepackage{pifont}
\usepackage{eurosym}
\usepackage{subcaption}
\usepackage{setspace}
\usepackage{bbm}
\usepackage{dsfont}
\usepackage{url}
\usepackage{amsmath,blkarray,booktabs, bigstrut}
\usepackage{multirow,multicol}
\usepackage{xcolor}
\usepackage{color}
\usepackage{imakeidx}
\makeindex[]

\usepackage{algorithm}
\usepackage{algpseudocode}
\algrenewcommand\algorithmicrequire{\textbf{Input:}}
\algrenewcommand\algorithmicensure{\textbf{Output:}}
\usepackage{relsize}
\usepackage{float}
\usepackage{tikz,pgf}
\usepackage{tikz-cd}
\usepackage{wasysym}
\usepackage{graphicx}
\usepackage{fancyhdr}
\usepackage{stmaryrd}
\usepackage{verbatim}

\usepackage{pgfplots}
\pgfplotsset{compat=1.15}
\usetikzlibrary{arrows}
\usepackage{xspace}

\usepackage[all,dvips,knot,web,arc,curve,color,frame]{xy}
\usepackage[all,knot,arc,color,web]{xy}
\xyoption{arc}
\xyoption{web}
\xyoption{curve}

\numberwithin{equation}{section}

\theoremstyle{plain}
\newtheorem{theorem}{Theorem}[section]
\newtheorem{proposition}[theorem]{Proposition}
\newtheorem{corollary}[theorem]{Corollary}
\newtheorem{lemma}[theorem]{Lemma}
\newtheorem{notation}[theorem]{Notation}

\theoremstyle{definition}
\newtheorem{definition}[theorem]{Definition}

\newtheorem{remark}[theorem]{Remark}

\newtheorem{question}[theorem]{Question}
\newtheorem{assumption}[theorem]{Assumption}

\usepackage[bookmarksnumbered=true]{hyperref} 
\hypersetup{
     colorlinks = true,
     linkcolor = blue,
     anchorcolor = blue,
     citecolor = teal,
     filecolor = blue,
     urlcolor = blue
     }
\newcommand\restr[2]{{
  \left.\kern-\nulldelimiterspace 
  #1 
  \vphantom{\big|} 
  \right|_{#2} 
  }}

\everymath{\displaystyle}
\allowdisplaybreaks

\makeatletter
\def\mathcenterto#1#2{\mathclap{\phantom{#1}\mathclap{#2}}\phantom{#1}}
\let\old@widetilde\widetilde
\def\widetildeto#1#2{\mathcenterto{#2}{\old@widetilde{\mathcenterto{#1}{#2\,}}}}
\let\old@widehat\widehat
\def\widehatto#1#2{\mathcenterto{#2}{\old@widehat{\mathcenterto{#1}{#2\,}}}}
\makeatother
\usepackage{graphicx}

\newcommand{\abs}[1]{\left| #1 \right|} 

\newcommand{\size}[1]{\left| #1 \right|} 
\newcommand{\pare}[1]{\left( #1 \right)} 
\newcommand{\set}[1]{{\left\{ #1 \right\}}} 

\newcommand*\closure[1]{\overline{#1}}

\newcommand{\JK}{\textup{v}_{J_{\mathcal{K}_{n}}}(J_{G})}
\newcommand{\VPJ}{\textup{v}_{P_{\mathcal{S}}}(J_G)}
\newcommand{\DG}{\mathcal{D}_c(G)}
\newcommand{\JS}{J_{\mathcal{T}(\mathcal{S})}}
\newcommand{\GS}{\mathcal{G}_{\mathcal{S}}}
\newcommand{\TS}{\mathcal{T}(\mathcal{S})}
\newcommand{\DMS}{\mathcal{D}(M(\mathcal{S}))}

\newcommand{\DVV}{\mathcal{D}_{c}(V_{1},V_{2})}
\newcommand{\VCN}{\textup{v}_{P_{\mathcal{S}}}(J_{C_{n}})}
\newcommand{\VCNN}{\textup{v}(J_{C_{n}})}

\newcommand{\Ccal}{\mathcal{C}}
\newcommand{\CS}{\mathcal{S}}

\newcommand{\vs}{\vspace*{.5em}}

\DeclareMathOperator{\rank}{rk}

\DeclareMathOperator{\supp}{supp}

\DeclareMathOperator{\CC}{\mathbb{C}}
\DeclareMathOperator{\VV}{\mathbb{V}}

\hypersetup{
    colorlinks=true,
    linkcolor=blue,
    filecolor=magenta,      
    urlcolor=cyan,
    pdftitle={Overleaf Example},
    pdfpagemode=FullScreen,
    }

\newtheorem{theoremA}{Theorem}

\newtheorem{theoremB}{Theorem}

\newtheorem{theoremC}{Theorem}

\newtheorem{theoremD}{Theorem}

\newtheorem{strategy}{Strategy}

\title{\vspace*{-2em} The V-Number of Binomial Edge Ideals: Minimal Cuts and Cycle Graphs}

\author{Emiliano Liwski}

\date{\vspace*{-.5em}\today\vspace*{-1.5em}}

\begin{document}
\maketitle

\begin{abstract}
The v-number of a graded ideal is an invariant recently introduced in the context of coding theory, particularly in the study of Reed--Muller-type codes. In this work, we study the localized v-numbers of a binomial edge ideal $J_G$ associated to a finite simple graph $G$. We introduce a new approach to compute these invariants, based on the analysis of transversals in families of subsets arising from dependencies in certain rank-two matroids. This reduces the computation of localized v-numbers to the determination of the radical of an explicit ideal and provides upper bounds for these invariants. Using this method, we explicitly compute the localized v-numbers of $J_G$ at the associated minimal primes corresponding to minimal cuts of $G$. Additionally, we determine the v-number of binomial edge ideals for cycle graphs and give an almost complete answer to a conjecture from \cite{dey2024v}, showing that the v-number of a cycle graph $C_n$ is either $\textstyle \left\lceil \frac{2n}{3} \right\rceil$ or $\textstyle \left\lceil \frac{2n}{3} \right\rceil - 1$.
\end{abstract}

\section{Introduction}

Let $S = \mathbb{K}[x_{1}, \ldots, x_{n}] = \oplus_{d=0}^{\infty} S_{d}$ be the polynomial ring in $n$ variables over a field $\mathbb{K}$, endowed with the standard grading. For a graded ideal $I \subset S$, the set of associated prime ideals, denoted by $\textup{Ass}(I)$ or $\textup{Ass}(S/I)$, consists of all prime ideals of the form $(I : f)$ for some homogeneous element $f \in S_{d}$. In \cite{cooper2020generalized}, the authors introduced a new invariant for graded ideals of $S$, called the \emph{\textup{v}-number}, in the context of their study of Reed–Muller-type codes.

\begin{definition}\textup{\cite[Definition~4.1]{cooper2020generalized}}\label{first definition}
Given a homogeneous ideal $I$ of $S$, the v-number of $I$ is the following invariant

\[\textup{v}(I):=\min\{d\geq 0: \exists f\in S_{d}\ \text{and $\mathfrak{p}\in \text{Ass}(I)$ with $(I:f)=\mathfrak{p}$}\}.\]

We define the v-number of $I$ locally at each associated prime $\mathfrak{p}$ of $I$ as:

\[\textup{v}_{\mathfrak{p}}(I):=\min\{d\geq 0: \exists f\in S_{d}\ \text{with $(I:f)=\mathfrak{p}$}\}.\]

Then $\textup{v}(I)=\min\{\textup{v}_{\mathfrak{p}}(I):\mathfrak{p}\in \text{Ass}(I)\}$. We call $\textup{v}_{\mathfrak{p}}(I)$ the localization of $\textup{v}(I)$ at $\mathfrak{p}$.
\end{definition}

The v-number of $I$ offers insight into the asymptotic behavior of the minimum distance function $\delta_{I}$ associated with projective Reed–Muller-type codes \cite{cooper2020generalized,pinto2023graph}. These codes play a fundamental role in error correction and reliable data transmission. Independently of its interpretation in coding theory, the minimum distance function $\delta_I$ is defined for any graded ideal in terms of the Hilbert–Samuel multiplicity of $S/I$; see~\cite{nunez2017footprint}. Furthermore, the local notion of v-numbers generalizes the concept of the degree of a point in a finite set of projective points, as introduced in~\cite{geramita1993cayley}.

Although initially inspired by questions in coding theory, the v-number has since been the subject of extensive study from algebraic and combinatorial perspectives. A significant body of work has focused on the v-number of monomial ideals (see \cite{biswas2024study,civan2023v,saha2024v,saha2022v,jaramillo2021v}). The first article devoted entirely to this invariant is \cite{jaramillo2021v}, where the authors investigated the v-number of edge ideals, providing a combinatorial formula for the v-number of squarefree monomial ideals. In particular, their formula admits a natural interpretation in terms of the face structure of an associated simplicial complex. In addition, several recent works have examined the asymptotic behavior of the v-number; see \cite{conca2024note,ficarra2025asymptotic,kumar2025slope}.

Several other works have explored the relationship between the v-number of $I$ and the Castelnuovo--Mumford regularity of $S/I$. In \cite{jaramillo2021v}, the authors establish that, in many relevant cases, $\textup{v}(I)$ is a lower bound for the regularity of $S/I$. Furthermore, in \cite{saha2022v}, it is shown that for certain classes of graphs, one has $\textup{v}(I) \leq \textup{im}(G) \leq \textup{reg}(S/I)$, where $\textup{im}(G)$ denotes the induced matching number of $G$ and $I$ is the edge ideal associated to $G$.

In general, computing the v-number of a graded ideal is a challenging task, especially in the case of non-monomial ideals. In this work, we investigate the v-number of binomial edge ideals.

\begin{definition}\label{definition binomial edge}
Let $S=\mathbb{K}[x_1,\ldots, x_n, y_1,\ldots, y_n]$, and let $G$ be a simple
graph on the vertex set $V(G)=[n]$ with set of
edges $E(G)$. The {\em binomial edge ideal} of $G$, denoted by $J_{G}$, is the quadratic binomial ideal defined by:

\[J_{G}:=(f_{i,j}:=x_{i}y_{j}-x_{j}y_{i}:\{i,j\}\in E(G),i<j).\]

In other words, $J_{G}$ is the ideal generated by the $2$-minors of the generic matrix
\begin{equation}\label{matrix X}X=\begin{pmatrix}
x_{1} & x_{2} & x_{3} & \cdots & x_{n}\\
y_{1} & y_{2} & y_{3} & \cdots & y_{n}
\end{pmatrix}\end{equation}
corresponding to the pairs of columns indexed by the edges of $G$. 
\end{definition}

Binomial edge ideals were introduced by Herzog, Hibi, Hreinsdottir, Kahle, and Rauh \cite{herzog2010binomial}, and by Ohtani \cite{ohtani2011graphs}, as a means of associating algebraic structures to graphs. 
Today, these ideals represent an active research topic in combinatorial commutative algebra and have also found important applications in algebraic statistics, where they appear as ideals generated by conditional independence statements (see \cite{herzog2010binomial}).

In this work, we focus on computing the v-numbers of a binomial edge ideal $J_G$ at its minimal primes, where $G$ is a simple graph on the vertex set $[n]$. Compared to the monomial case, this represents a more challenging problem. However, since the v-number of $J_G$ is additive over the connected components of $G$ (see~\textup{\cite[Proposition~3.1]{jaramillo2024connected}}), it suffices to restrict our attention to the case where $G$ is connected. In this setting, the minimal primes of $J_G$ correspond bijectively to the \emph{minimal $k$-cuts} of $G$ together with the empty set; see Definition~\ref{minimal cuts} and Theorem~\ref{Theorem minimal primes}. For each subset $\mathcal{S} \subset [n]$ that is either empty or a minimal $k$-cut, we denote by $P_{\mathcal{S}}$ the corresponding minimal prime, as defined in Definition~\ref{PS}. The main question addressed in this work is the following:

\begin{question}\label{question 2}
Given a connected graph $G$, determine the localized v-number $\textup{v}_{P_{\mathcal{S}}}(J_G)$ of its associated binomial edge ideal $J_G$ at each minimal prime $P_{\mathcal{S}} \in \textup{Ass}(J_G)$.
\end{question}

In \cite{dey2024v}, this question was studied for several classes of graphs, including Cohen--Macaulay closed graphs, path graphs, cycles, and binary trees. Furthermore, Question~\ref{question 2} was independently resolved for the case $\mathcal{S} = \emptyset$ in \cite{ambhore2024v} and \cite{jaramillo2024connected}, where the authors provided a combinatorial interpretation of the invariant in terms of the connected domination number of the graph. In this work, we aim to address this question in two specific settings: \begin{enumerate} 
\item For any connected graph $G$ and any minimal cut $\mathcal{S}$ of $G$. 
\item For $G$ equal to the cycle graph $C_n$. 
\end{enumerate}

\smallskip

\noindent {\bf Our contributions.} We now summarize the main contributions of this work. Throughout, let $G$ be a simple connected graph on $[n]$, and let $\min(G)$ denote the collection of all minimal $k$-cuts of $G$, together with the empty set. We denote by $V(G)$ and $E(G)$ the vertex set and edge set of $G$, respectively. In Section~\ref{section transversals}, we introduce a new approach to Question~\ref{question 2}, which relies on constructing an explicit ideal defined by transversals of a family of subsets arising from the dependencies of certain rank-two matroids. This formulation allows us to reduce Question~\ref{question 2} to the problem of determining the radical of this ideal. For clarity, we begin by introducing the necessary definitions. We refer the reader to Subsection~\ref{matroids} for the relevant background on matroids.

\begin{definition}\label{definition MS}
Let $G$ be a simple graph on $[n]$, and let $\mathcal{S} \subset [n]$. Denote by $G_{1}, \ldots, G_{c(\mathcal{S})}$ the connected components of the induced subgraph $G_{[n] \setminus \mathcal{S}}$. We define the matroid $M(\mathcal{S})$ on $[n]$ of rank at most two by specifying its dependent sets:
\[
\mathcal{D}(M(\mathcal{S})) = \{ B \subseteq [n] : B \cap \mathcal{S} \neq \emptyset \} \cup_{i\in c(\mathcal{S})} \{ B : |B \cap V(G_{i})| \geq 2\} \cup \{ B : |B| \geq 3 \}.
\]
\end{definition}

\begin{definition}(Definition~\ref{important notation})
For a fixed $\mathcal{S} \in \min(G)$, we define $\Delta_{\mathcal{S}}$ as the collection
\[
\left\{ \mathcal{D}(M(\mathcal{S}')) \setminus \mathcal{D}(M(\mathcal{S})) : \mathcal{S}' \in \min(G) \setminus \{\mathcal{S}\} \right\} 
\subseteq \textstyle\left( \binom{[n]}{1} \cup \binom{[n]}{2} \right) \setminus \mathcal{D}(M(\mathcal{S})).
\]
We denote by $\mathcal{T}(\mathcal{S})$ the collection of all transversals of $\Delta_{\mathcal{S}}$. Each $A \in \mathcal{T}(\mathcal{S})$ can be written uniquely as a disjoint union $A = A_1 \amalg A_2$, where 
\[
A_1 \subseteq \textstyle \binom{[n]}{1} \quad \text{and} \quad A_2 \subseteq \binom{[n]}{2}.
\]
We identify $A_1$ naturally with a subset of $[n]$.
\end{definition}

We now introduce the ideal associated to transversals, as mentioned above.

\begin{definition}(Definitions~\ref{important notation} and~\ref{notation JS})
Let $A \in \mathcal{T}(\mathcal{S})$ and consider a partition $\{C, D\}$ of $A_{1}$. We define
\[
g_{A,C,D} := \prod_{k \in C} x_k   \prod_{k \in D} y_k   \prod_{\{i,j\} \in A_2} (x_{i}y_{j}-x_{j}y_{i}).
\]
The ideal associated to $\mathcal{T}(\mathcal{S})$ is then defined by
\[
J_{\mathcal{T}(\mathcal{S})} := \left( g_{A,C,D} \ : \ A \in \mathcal{T}(\mathcal{S}), \ C \amalg D = A_1 \right).
\]
\end{definition}

In Theorem~\ref{theorem transversal}, we show that the saturation $(J_{G} : P_{\mathcal{S}})$ coincides with the radical of $J_{G} + J_{\mathcal{T}(\mathcal{S})}$.
The following theorem is the main result of Section~\ref{section transversals}, where we apply the method of transversals of matroid dependencies. This provides a new framework for computing v-numbers or for deriving effective upper bounds.

\begin{theoremA} (Theorem~\ref{corollary radical}) \label{corollary radical 2}
With the notation above, the localized v-number satisfies
\[
\textup{v}_{P_{\mathcal{S}}}(J_{G}) \leq \min \{ |A_{1}| + 2|A_{2}| : A \in \mathcal{T}(\mathcal{S}) \}.
\]
Moreover, equality holds if the ideal $J_{G}+J_{\mathcal{T}(\mathcal{S})}$ is radical.
\end{theoremA}

Even when $J_{G} + J_{\mathcal{T}(\mathcal{S})}$ is not radical, this approach still offers an effective strategy for computing $\textup{v}_{P_{\mathcal{S}}}(J_G)$.

\begin{strategy}(Remark~\ref{remark strategy})\label{strategy}
To compute $\textup{v}_{P_{\mathcal{S}}}(J_G)$, we proceed as follows:
\begin{enumerate}
\item[{\rm (i)}] Determine the set of all transversals $\mathcal{T}(\mathcal{S})$.
\item[{\rm (ii)}] Compute the radical of $J_{G} + J_{\mathcal{T}(\mathcal{S})}$, typically via the computation of a Gröbner basis.
\item[{\rm (iii)}] Identify the minimal degree of a homogeneous polynomial in $\textstyle \sqrt{J_{G} + J_{\mathcal{T}(\mathcal{S})}}/J_G$. By Theorems~\ref{results v number} and~\ref{theorem transversal}, this value equals $\textup{v}_{P_{\mathcal{S}}}(J_G)$.
\end{enumerate}
\end{strategy}

In \cite{ambhore2024v} and \cite{jaramillo2024connected}, the authors provided a combinatorial interpretation of the invariant $\textup{v}_{J_{\mathcal{K}_{n}}}(J_G)$, corresponding to the case $\mathcal{S} = \emptyset$, where $\mathcal{K}_{n}$ denotes the complete graph on $[n]$, in terms of the connected domination number of the graph. We now recall this notion.

\begin{definition}\label{definition connected domination}
Let $G$ be a connected graph. A \emph{connected dominating set} of $G$ is a subset $B \subseteq V(G)$ such that the induced subgraph $G_B$ is connected and every vertex in $V(G) \setminus B$ is adjacent to some vertex in $B$. We denote the collection of connected dominating sets of $G$ by $\mathcal{D}_c(G)$, and define the \emph{connected domination number} of $G$ as
\[
\gamma_c(G) := \min\{ |B| : B \in \mathcal{D}_c(G) \}.
\]
\end{definition}

By applying our new approach described in Strategy~\ref{strategy} to the case $\mathcal{S} = \emptyset$, we recover the result previously obtained in \cite{ambhore2024v} and \cite{jaramillo2024connected}.

\begin{theoremB}(Theorem~\ref{theorem empty})
Let $G$ be a connected graph on $[n]$. Then, we have
\[\JK=\begin{cases}
\gamma_{c}(G) & \text{if $G\neq \mathcal{K}_{n},$}\\
0 & \text{if $G=\mathcal{K}_{n}$.}
\end{cases}
\]
\end{theoremB}

We now turn our attention to Question~\ref{question 2} in the case where $\mathcal{S}$ is a minimal $2$-cut, which, for simplicity, we shall refer to simply as a minimal cut; that is, a subset of vertices whose removal disconnects the graph, and which is minimal with respect to this property. For such a set $\mathcal{S}$, we compute $\VPJ$ explicitly. Moreover, we obtain a combinatorial interpretation of this invariant in terms of connected dominance. However, this does not coincide with the classical connected domination number, but rather introduces a new graph-theoretic notion of dominance specifically adapted to the setting of minimal cuts, which we proceed to define.

\begin{definition}(Definition~\ref{definition gACDij 21})
Let $\mathcal{S} \subset [n]$ be a minimal cut of $G$, and denote by $V_{1}$ and $V_{2}$ the vertex sets of the connected components of the induced subgraph on $[n] \setminus \mathcal{S}$. We define
\[
\mathcal{D}_{c}(V_{1}, V_{2}) = \left\{ A \subseteq V_{1} \cup V_{2} \,\middle|\, A \cap V_{i} \in \mathcal{D}_{c}\big(G_{V_{i} \cup \mathcal{S}}\big) \text{ for each } i = 1, 2 \right\},
\]
where $G_{V_{i} \cup \mathcal{S}}$ denotes the induced subgraph on $V_{i} \cup \mathcal{S}$. We then set
\[
\gamma_{c}(V_{1}, V_{2}) := \min \left\{ |A| : A \in \mathcal{D}_{c}(V_{1}, V_{2}) \right\}.
\]
\end{definition}

By applying our approach from Strategy~\ref{strategy} to the case where $\mathcal{S}$ is a minimal cut of $G$, we obtain the following result.

\begin{theoremC}(Theorem~\ref{main theorem})
Let $G$ be a connected graph on $[n]$ and let $\mathcal{S} \subset [n]$ be a minimal cut of $G$. Then,
\[
\VPJ = \gamma_{c}(V_{1}, V_{2}) = \min\left\{ \lvert A \rvert : A \in \mathcal{D}_{c}(V_{1}, V_{2}) \right\}.
\]
\end{theoremC}

We also address the computation of the v-number in the case where $G$ is the cycle graph $C_n$, see Definition~\ref{definition cycle graph}. It was established in~\textup{\cite[Corollary~4.10]{dey2024v}} that $\mathrm{v}(J_{C_n}) \leq \textstyle \left\lceil \frac{2n}{3} \right\rceil$, and it was conjectured in the same work that the equality holds for $n \geq 6$. By applying our approach from Strategy~\ref{strategy}, together with several additional techniques tailored to this setting, we provide an almost complete confirmation of this conjecture: we prove that $\VCNN$ is either $\textstyle \left\lceil \frac{2n}{3} \right\rceil$ or $\textstyle \left\lceil \frac{2n}{3} \right\rceil - 1$.

\begin{theoremD}(Theorem~\ref{main theorem 2})
The v-number of $C_n$ satisfies the following:
\begin{enumerate}
\item[{\rm (i)}] If $n \equiv 0 \pmod{3}$, then $\VCNN = \textstyle \frac{2n}{3}$.
\item[{\rm (ii)}] If $n \equiv 1, 2 \pmod{3}$, then $\VCNN$ is either $\textstyle \left\lceil \frac{2n}{3} \right\rceil$ or $\textstyle \left\lceil \frac{2n}{3} \right\rceil - 1$.
\end{enumerate}
\end{theoremD}

\smallskip

\noindent
\textbf{Outline.} Section~\ref{preliminaries} reviews the necessary background on binomial edge ideals, v-numbers, and matroids. In Section~\ref{section transversals}, we present a new approach to compute v-numbers, based on identifying transversals in families of subsets arising from dependencies in certain rank-two matroids. In Section~\ref{section minimal cuts}, we explicitly compute the localized v-numbers of $J_G$ at the minimal primes corresponding to minimal cuts of $G$. Finally, in Section~\ref{section v numbers of cycle graphs}, we establish that the v-number of the cycle graph $C_n$ is either $\textstyle \left\lceil \frac{2n}{3} \right\rceil$ or $\textstyle \left\lceil \frac{2n}{3} \right\rceil - 1$.

\section{Preliminaries}\label{preliminaries}

In this section, we gather several well-known properties and results concerning binomial edge ideals, v-numbers and matroids that will be relevant in the subsequent sections. For simplicity, we work over the field of complex numbers $\mathbb{C}$ throughout.

\subsection{Primary decomposition of binomial edge ideals}

Let $G$ be a simple graph on $[n]$, and let $J_{G}$ denote its binomial edge ideal, as introduced in Definition~\ref{definition binomial edge}. When $G=\mathcal{K}_{n}$ is the complete graph on $n$ vertices, the binomial edge ideal $J_{\mathcal{K}_{n}}$ coincides with the ideal of $2$-minors of the matrix $X$ from Equation~\eqref{matrix X}. This observation highlights that binomial edge ideals naturally generalize classical determinantal ideals.

In \cite{herzog2010binomial} the authors prove that binomial edge ideals are radical and they give a combinatorial description of their associated primes. 

\begin{definition}\label{PS}
Given a subset $\mathcal{S}\subset [n]$, we define a prime ideal $P_{\mathcal{S}}$ as follows. Denote by $G_{1},\ldots,G_{c(\mathcal{S})}$ the connected components of the induced subgraph $G_{[n]\setminus \mathcal{S}}$. For each $G_{i}$, we denote by $\widetilde{G}_{i}$ the complete graph on $V(G_{i})$ and we set
\[P_{\mathcal{S}}:=(x_{i},y_{i}:i\in \mathcal{S})+(J_{\widetilde{G}_{1}},\ldots,J_{\widetilde{G}_{c(S)}}).\]
\end{definition}

This ideal is prime and all the associated primes of $J_{G}$ are of this form:

\begin{theorem}\textup{\cite[Theorem~3.2]{herzog2010binomial}}\label{theorem decomposition}
Let $G$ be a simple graph on $[n]$. Then $J_{G}=\cap_{\mathcal{S}\subset [n]}P_{\mathcal{S}}$.
\end{theorem}

Throughout the paper, we assume that the graph $G$ is connected.

\begin{definition}\label{minimal cuts}
A vertex $i \in [n]$ is called a \emph{cut point} of a graph $H$ if $H$ has fewer connected components than $H_{[n] \setminus \{i\}}$.  

For a graph $G$ on $[n]$, a nonempty subset $\mathcal{S} \subsetneq [n]$ is said to be a \emph{minimal $k$-cut} of $G$ if the induced subgraph $G_{[n] \setminus \mathcal{S}}$ has exactly $k$ connected components, and each $i \in \mathcal{S}$ is a cut point of the graph $G_{([n] \setminus \mathcal{S}) \cup \{i\}}$.  
In the case $k = 2$, we simply refer to $\mathcal{S}$ as a \emph{minimal cut}.
\end{definition}

As the following result shows, among all $\mathcal{S}$, the minimal primes of $J_{G}$ are those corresponding to minimal $k$-cuts.

\begin{theorem}\textup{\cite[Corollary~3.9]{herzog2010binomial}}\label{Theorem minimal primes}
Let $G$ be a connected graph on $[n]$, and $\mathcal{S}\subset [n]$. Then $P_{\mathcal{S}}$ is a minimal prime if and only if $\mathcal{S}=\emptyset$ or $\mathcal{S}$ is a minimal $k$-cut for some $k\geq 2$.
\end{theorem}

\subsection{V-numbers}

Let $S=\CC[x_{1},\ldots,x_{n}]$ be the polynomial ring in $n$ variables with the standard grading. Recall the definition of the v-number $\textup{v}(I)$ from Definition~\ref{first definition}. Our primary focus is to investigate Question~\ref{question 2}. To this end, we collect several results from \cite{grisalde2021induced} that will be fundamental for our purposes.

\begin{theorem}\textup{\cite[Theorem~10]{grisalde2021induced}}\label{results v number}
Let $I\subset S$ be a homogeneous ideal and let $\mathfrak{p}\in \textup{Ass}(I)$. The following hold.
\begin{enumerate}
\item[{\rm (i)}] If $\mathcal{G}=\{\overline{g}_{1},\ldots,\overline{g}_{r}\}$ is a homogeneous minimal generating set of $(I:\mathfrak{p})/I$, then
\[\textup{v}_{\mathfrak{p}}(I)=\min\{\deg(g_{i}):i\in [r]\ \text{and $(I:g_{i})=\mathfrak{p}$}\}.\]
\item[{\rm (ii)}] If $I$ has no embedded primes, then
\[
\textup{v}_{\mathfrak{p}}(I) = \alpha\left( (I : \mathfrak{p}) / I \right),
\]
where for any module $M$, we define $\alpha(M) := \min\{ \deg(f) : f \in M \setminus \{0\} \}$.
\end{enumerate}
\end{theorem}

\subsection{Matroids}\label{matroids}

As in this work we only work with matroids of rank at most two, all definitions will be presented within this setting. For the general theory of matroids, see \cite{Oxley, piff1970vector}.

\begin{definition}
A matroid $M$ of rank at most two consists of a ground set $[n]$, together with a collection $\mathcal{I}$ of subsets of $[n]$ of cardinality at most two, called \emph{independent sets}, which satisfy the following axioms: the empty set is independent, i.e., $\emptyset \in \mathcal{I}$; the hereditary property, meaning if $I \in \mathcal{I}$ and $I' \subset I$, then $I' \in \mathcal{I}$; and the exchange property, which states that if $I_1, I_2 \in \mathcal{I}$ with $|I_1| < |I_2|$, then there exists an element $e \in I_2 \setminus I_1$ such that $I_1 \cup \{e\} \in \mathcal{I}$.
\end{definition}

\begin{definition}\label{def:dependant}
Let $M$ be a matroid of rank at most two on $[n]$. We introduce the following notions:

\begin{itemize}
\item A subset of $[n]$ is called \emph{dependent} if it is not independent. The collection of all dependent sets of $M$ is denoted by $\mathcal{D}(M)$.


\item An element $x \in [d]$ is called a \emph{loop} if $\{x\} \in \mathcal{D}(M)$.
\item Two distinct elements $x, y \in [d]$ are said to be \emph{parallel}, or form a \emph{double point}, if $\{x,y\} \in \mathcal{D}(M)$, while both $\{x\}$ and $\{y\}$ are independent.
\end{itemize}
\end{definition}

\begin{definition}\normalfont\label{uniform 3}
The uniform matroid $U_{1, n}$ on the ground set $[n]$ of rank $1$ is  
is defined as follows: a subset $S\subset [d]$ is independent if and only if $|S| \leq 1$.
\end{definition}

Let $G$ be a simple graph on $[n]$. To each subset $\mathcal{S} \subset [n]$, we naturally associate a matroid $M(\mathcal{S})$ on $[n]$ of rank at most two, as described in Definition~\ref{definition MS}, whose collection of dependent sets is given by
\[
\mathcal{D}(M(\mathcal{S})) = \{ B \subseteq [n] : B \cap \mathcal{S} \neq \emptyset \} \cup_{i\in c(\mathcal{S})} \{ B : |B \cap V(G_{i})| \geq 2\} \cup \{ B : |B| \geq 3 \},
\]
where $G_{1}, \ldots, G_{c(\mathcal{S})}$ are the connected components of the induced subgraph $G_{[n] \setminus \mathcal{S}}$. As we shall see, these matroids are crucial in our study of the v-number of binomial edge ideals.

\begin{remark}\label{connected and uniform}
Note that if $G$ is connected, then the matroid $M(\emptyset)$ coincides with the uniform matroid $U_{1,n}$ of rank one.
\end{remark}

\section{Transversals and v-numbers of binomial edge ideals}\label{section transversals}

This section is devoted to the study of 
v-numbers of binomial edge ideals. By analyzing transversals of a family of subsets arising from the dependencies of certain rank-two matroids, we reduce the problem of computing localized 
v-numbers to that of determining the radical of an explicit ideal. This perspective also yields upper bounds for the localized 
v-numbers. Throughout this section, we fix $G$ to be a simple and connected graph on $[n]$.

\subsection{Reducing the problem to computing the radical of a transversal ideal}\label{reduction with transversals}

In this subsection, we show how Question~\ref{question 2} reduces to computing the radical of a certain ideal generated by transversals. Throughout this section, we fix a minimal prime $P_{\mathcal{S}}$ of $J_{G}$, and focus on addressing Question~\ref{question 2} with respect to this prime. 
We begin by recalling the definition of a transversal and introduce some necessary notation.

\begin{definition}\label{transveral definition}
Let $\Delta$ be a collection of subsets of a set $E$. A subset $A \subset \cup_{e\in \Delta}e$ is called a \emph{transversal} of $\Delta$ if $A \cap e \neq \emptyset$ for all $e \in \Delta$.
\end{definition}

\begin{definition}\label{important notation}
We now introduce the notions and terminology used throughout this section. For context, recall Definitions~\ref{definition MS} and~\ref{def:dependant}.

\begin{enumerate}
    \item[{\rm (i)}] For each $k \geq 2$, let ${\min}_k(G)$ denote the set of minimal $k$-cuts of the graph $G$. We set \[\min(G) := \bigcup_{k \geq 2} {\min}_k(G)\cup \{\emptyset\},\] 
    the collection of all minimal $k$-cuts of $G$, together with the emptyset.

    \item[{\rm (ii)}] For a fixed $\mathcal{S} \in \min(G)$, define $\Delta_{\mathcal{S}}$ as the collection of subsets
    \[
    \left\{ \mathcal{D}(M(\mathcal{S}')) \setminus \mathcal{D}(M(\mathcal{S})) : \mathcal{S}' \in \min(G) \setminus \{ \mathcal{S} \} \right\}
    \subseteq \textstyle (\binom{[n]}{1} \cup \binom{[n]}{2})\setminus \mathcal{D}(M(\mathcal{S})).
    \]
    We denote by $\mathcal{T}(\mathcal{S})$ the set of all transversals of $\Delta_{\mathcal{S}}$.

    \item[{\rm (iii)}]\label{item 3} Each $A \in \mathcal{T}(\mathcal{S})$ is written as a disjoint union $A = A_1 \amalg A_2$, where \[A_1 \subseteq \textstyle \binom{[n]}{1}\quad  \text{and} \quad A_2 \subseteq \binom{[n]}{2}.\] We identify $A_1$ with a subset of $[n]$ in the natural way.

    \item[{\rm (iv)}] For arbitrary subsets $C, D \subseteq [n]$, define the monomial \[g_{C,D} := \prod_{k \in C} x_k \prod_{k \in D} y_k.\]

    \item[{\rm (v)}]\label{item v} For any $A \in \mathcal{T}(\mathcal{S})$ and a partition $A_1 = C \amalg D$, define \[g_{A,C,D} := g_{C,D} \cdot \prod_{\{i,j\} \in A_2} f_{i,j},\]
    where $f_{i,j}=x_{i}y_{j}-x_{j}y_{i}$.
\end{enumerate}
\end{definition}

We are now ready to prove one of the two main results of this section. 

\begin{theorem}\label{theorem transversal}
With the notation above, the following equality of ideals holds:
\begin{equation}\label{equality of ideals}(J_{G}:P_{\mathcal{S}})=\sqrt{J_{G}+(g_{A,C,D}:A\in \mathcal{T}(\mathcal{S}) \ \text{and $C\amalg D=A_{1}$})}.\end{equation}
\end{theorem}

\begin{proof}
Set
\[
J := (g_{A,C,D} : A \in \mathcal{T}(\mathcal{S}) \ \text{and} \ C \amalg D = A_1).
\]
Since both sides of~\eqref{equality of ideals} are radical ideals, it suffices to prove the equality of their varieties:
\begin{equation}\label{closure ideal}
\overline{\mathbb{V}(J_G) \setminus \mathbb{V}(P_{\mathcal{S}})} = \mathbb{V}(J_G) \cap \mathbb{V}(J),
\end{equation}
where we use the identity
\[
\mathbb{V}((J_G : P_{\mathcal{S}})) = \overline{\mathbb{V}(J_G) \setminus \mathbb{V}(P_{\mathcal{S}})},
\]
see~\textup{\cite[§4, Theorem~10 (iii)]{david1991ideals}}, and the fact that $J_{G}$ is radical, see \textup{\cite[Corollary~2.2]{herzog2010binomial}}. Moreover, by Theorems~\ref{theorem decomposition} and~\ref{Theorem minimal primes}, the left-hand side of~\eqref{closure ideal} equals
\begin{equation}\label{union of ideals}
\bigcup_{\mathcal{S}' \in \min(G) \setminus \{\mathcal{S}\}} \mathbb{V}(P_{\mathcal{S}'}).
\end{equation}

We begin by proving the inclusion $\subseteq$. Let $\gamma = (\gamma_1, \ldots, \gamma_n) \in (\CC^2)^n$ be a point in~\eqref{union of ideals}, and suppose without loss of generality that $\gamma \in \mathbb{V}(P_{\mathcal{S}'})$ for some $\mathcal{S}' \ne \mathcal{S}$. We aim to show that $\gamma \in \mathbb{V}(J)$.

Let $g_{A,C,D}$ be an arbitrary generator of $J$. Since $A \in \mathcal{T}(\mathcal{S})$, there exists an element $e \in A \cap \mathcal{D}(M(\mathcal{S}'))$. We distinguish two cases:

\smallskip
\noindent
\textbf{Case $e = \{i\}$:} Then $i \in \mathcal{S}'$, and since $\gamma \in \mathbb{V}(P_{\mathcal{S}'})$, we have $\gamma_i = 0$. Moreover, by construction of $g_{A,C,D}$, either $x_i$ or $y_i$ divides it, so $g_{A,C,D}(\gamma) = 0$.

\smallskip
\noindent
\textbf{Case $e = \{j,k\}$:} Then $j$ and $k$ lie in the same connected component of $G_{[n] \setminus \mathcal{S}'}$, and since $\gamma \in \mathbb{V}(P_{\mathcal{S}'})$, this implies that $\gamma_j$ and $\gamma_k$ are linearly dependent. Since $f_{j,k}$ divides $g_{A,C,D}$, it follows again that $g_{A,C,D}(\gamma) = 0$.

\smallskip
In both cases, $g_{A,C,D}$ vanishes at $\gamma$, and since this holds for all generators, we conclude $\gamma \in \mathbb{V}(J)$.

We now prove the reverse inclusion. Let $\gamma \in \mathbb{V}(J_G) \cap \mathbb{V}(J)$ and suppose, for contradiction, that $\gamma \notin \mathbb{V}(P_{\mathcal{S}'})$ for every $\mathcal{S}' \in \min(G) \setminus \{ \mathcal{S} \}$. We can assume $\gamma \in \mathbb{V}(P_{\mathcal{S}})$, because on the contrary we would have $\gamma \in \mathbb{V}(J_{G})\setminus \mathbb{V}(P_{\mathcal{S}})$.

Let $M(\gamma)$ be the matroid on $[n]$ induced by the vectors $\gamma_1, \ldots, \gamma_n$. Then, since $\gamma \in \mathbb{V}(P_{\mathcal{S}})$ and $\gamma \notin \mathbb{V}(P_{\mathcal{S}'})$ for every $\mathcal{S}' \in \min(G) \setminus \{ \mathcal{S} \}$, it follows that
\[
\mathcal{D}(M(\gamma)) \supseteq \mathcal{D}(M(\mathcal{S})) \quad \text{and} \quad 
\mathcal{D}(M(\gamma)) \nsupseteq \mathcal{D}(M(\mathcal{S}'))
\]
for all $\mathcal{S}' \in \min(G) \setminus \{ \mathcal{S} \}$.

Let $A$ be a set containing one element from each difference
\[
\mathcal{D}(M(\mathcal{S}')) \setminus \mathcal{D}(M(\gamma)).
\]
By construction, $A \in \mathcal{T}(\mathcal{S})$. Since $A \subseteq \textstyle \binom{[n]}{1} \cup \binom{[n]}{2}$ and all its elements lie outside $\mathcal{D}(M(\gamma))$, we have:
\begin{itemize}
\item For each $i \in A_1$, at least one of $x_i$ or $y_i$ does not vanish at $\gamma$;
\item For each $\{j,k\} \in A_2$, we have $f_{j,k}(\gamma) \ne 0$.
\end{itemize}

Hence, we can find a partition $C \amalg D = A_1$ such that $g_{A,C,D}(\gamma) \ne 0$, contradicting the assumption $\gamma \in \mathbb{V}(J)$. This contradiction completes the proof.
\end{proof}

\begin{definition}\label{notation JS}
To simplify notation, we denote by
\[
J_{\mathcal{T}(\mathcal{S})} := \left(g_{A,C,D} : A \in \mathcal{T}(\mathcal{S}) \ \text{with} \ C \amalg D = A_1\right)
\]
the ideal generated by all such polynomials.
\end{definition}

As a consequence of Theorem~\ref{theorem transversal}, we obtain upper bounds for the localized v-numbers. These bounds are computed in a purely combinatorial way.

\begin{theorem}\label{corollary radical}
With the notation above, the localized \textup{v}-number satisfies
\[
\textup{v}_{P_{\mathcal{S}}}(J_{G}) \leq \min \{ |A_{1}| + 2|A_{2}| : A \in \mathcal{T}(\mathcal{S}) \}.
\]
Moreover, equality holds if the ideal $J_{G}+J_{\mathcal{T}(\mathcal{S})}$ is radical.
\end{theorem}

\begin{proof}
Let $A \in \mathcal{T}(\mathcal{S})$. By Theorem~\ref{theorem transversal}, we have $g_{A,C,D} \in (J_G : P_{\mathcal{S}})$ for any partition $\{C, D\}$ of $A_1$. Moreover, since all elements of $A$ lie outside $\mathcal{D}(M(\mathcal{S}))$, each factor of $g_{A,C,D}$ is nonzero on $\mathbb{V}(P_{\mathcal{S}})$. As $\mathbb{V}(P_{\mathcal{S}})$ is irreducible, 
It follows that $g_{A,C,D}$ does not vanish identically on $\mathbb{V}(P_{\mathcal{S}})$, and therefore not on $\mathbb{V}(J_G)$ either. Consequently, we have $g_{A,C,D} \notin J_G$. Note also that
\[
\deg(g_{A,C,D}) = |A_1| + 2|A_2|.
\]
Thus, by Theorem~\ref{results v number} (ii), we have 
\[\textup{v}_{P_{\mathcal{S}}}(J_G)=\alpha((J_G : P_{\mathcal{S}})/J_{G})\leq \deg(g_{A,C,D}) = |A_1| + 2|A_2|,\]
establishing the first part of the claim, since $A\in \TS$ is arbitrary.

Now assume that $J_{G} + \JS$ is radical. By Theorem~\ref{theorem transversal}, we have
\[
(J_G : P_{\mathcal{S}}) / J_G = (J_G + \JS)/J_G.
\]
Moreover, 
\[
\left\{ g_{A,C,D} : A \in \mathcal{T}(\mathcal{S}) \ \text{with} \ C \amalg D = A_1 \right\}
\]
is a generating set for this quotient. Then, applying Theorem~\ref{results v number} (i), we obtain
\[
\textup{v}_{P_{\mathcal{S}}}(J_G) \geq \min \left\{ \deg(g_{A,C,D}) : A \in \mathcal{T}(\mathcal{S}),\ C \amalg D = A_1 \right\} 
= \min \left\{ |A_{1}| + 2|A_{2}| : A \in \mathcal{T}(\mathcal{S}) \right\},
\]
which completes the proof.
\end{proof}

\begin{remark}\label{remark strategy}
This method reduces the computation of the localized v-number to determining a finite set of generators for the radical of $J_{G} + \JS$. Once such generators are identified, Theorem~\ref{results v number} provides an explicit formula for $\VPJ$. This leads to the following strategy for computing $\VPJ$:

\begin{enumerate}
\item[{\rm (i)}] Determine the collection of all transversals $\mathcal{T}(\mathcal{S})$.
\item[{\rm (ii)}] Compute the radical of $J_{G} + J_{\mathcal{T}(\mathcal{S})}$, typically via Gröbner basis methods.
\item[{\rm (iii)}] Identify the minimal degree of a homogeneous polynomial in  $\textstyle \sqrt{J_{G} + J_{\mathcal{T}(\mathcal{S})}}/J_G$, which equals $\textup{v}_{P_{\mathcal{S}}}(J_G)$.
\end{enumerate}
\end{remark}

\subsection{Computing $\textup{v}_{J_{\mathcal{K}_{n}}}(J_{G})$}

In this subsection, we fix a connected graph $G$ and consider the case $\mathcal{S} = \emptyset$, corresponding to the prime $J_{\mathcal{K}_{n}}$, where $\mathcal{K}_{n}$ is the complete graph on $[n]$. Using the reduction method introduced in Subsection~\ref{reduction with transversals}, we compute the algebraic invariant $\JK$. We further show that this invariant coincides with a purely combinatorial invariant of $G$: its connected domination number, see Definition~\ref{definition connected domination}.

We note that the results established in this subsection also appear in \cite{jaramillo2024connected} and \cite{ambhore2024v}. However, our treatment is based on a substantially different method, providing a new proof of these results. 

Our first step in computing $\JK$ is to describe the set $\mathcal{T}(\emptyset)$ of transversals of
\[
    \left\{ \mathcal{D}(M(\mathcal{S}')) \setminus \mathcal{D}(U_{1,n}) : \mathcal{S}' \in \min(G) \setminus \{ \emptyset \} \right\},
\]
where Remark~\ref{connected and uniform} applies. To that end, we begin by proving the following combinatorial lemma, for which we recall the definition of $\mathcal{D}_c(G)$ from Definition~\ref{definition connected domination}.

\begin{lemma}\label{transversal and connected dominant}
Let $G\neq \mathcal{K}_{n}$ be a connected graph on the vertex set $[n]$. A subset $A \subset [n]$ is a transversal of the collection
\[
\Delta = \{ S : S \in \min(G) \setminus \{ \emptyset \} \}
\]
if and only if $A$ belongs to the collection of connected dominating sets $\mathcal{D}_c(G)$.
\end{lemma}

\begin{proof}
Assume first that $A \in \mathcal{D}_c(G)$. To show that $A$ is a transversal of $\Delta$, we must prove that $A \cap S \neq \emptyset$ for every $S \in \min(G) \setminus \{ \emptyset \}$. Suppose, for contradiction, that there exists $S \in \min(G) \setminus \{ \emptyset \}$ such that $A \cap S = \emptyset$. Then $A \subset [n] \setminus S$, and since $A$ is connected and dominant, it follows that the induced subgraph on $[n] \setminus S$ is connected. This contradicts the assumption that $S \in \min(G)\setminus \{ \emptyset \}$, as such sets have disconnected complements.

Conversely, suppose $A$ is a transversal of $\Delta$. We show that $A \in \mathcal{D}_c(G)$ by verifying that $A$ is both connected and dominant.

\smallskip

\textbf{Connectedness.} Suppose, for contradiction, that $A$ is disconnected. Then the removal of the vertices in $[n] \setminus A$ disconnects the graph. Hence, there exists a minimal subset $S \subset [n] \setminus A$ such that the induced subgraph on $[n] \setminus S$ is disconnected. By minimality, we have $S \in {\min}_{2}(G)$ and $S \cap A = \emptyset$, contradicting that $A$ is a transversal of $\Delta$.

\smallskip

\textbf{Dominance.} Suppose, for contradiction, that $A$ is not dominant. Then there exists a vertex $i \in [n] \setminus A$ that is not adjacent to any vertex in $A$, so the subgraph induced by $A \cup \{i\}$ is disconnected. It follows that the removal of $[n] \setminus (A \cup \{i\})$ disconnects $G$. Thus, there exists a minimal subset $S \subset [n] \setminus (A \cup \{i\})$ such that $[n] \setminus S$ is disconnected. Again, we find $S \in \min(G)$ and $S \cap A = \emptyset$, contradicting that $A$ is a transversal of $\Delta$. This completes the proof.
\end{proof}

We are now in a position to describe the set of transversals $\mathcal{T}(\emptyset)$ and to present an explicit finite generating set for the ideal $J_{\mathcal{T}(\emptyset)}$.

\begin{proposition}\label{ideal J emptyset}
A subset $A \subset \textstyle (\binom{[n]}{1} \cup \binom{[n]}{2})\setminus \mathcal{D}(M(\emptyset))$ belongs to $\mathcal{T}(\emptyset)$ if and only if $A_{1} \in \mathcal{D}_{c}(G)$. Consequently, we have
\[
J_{\mathcal{T}(\emptyset)} = \left( g_{C,D} : C \cup D \in \mathcal{D}_{c}(G),\ C \cap D = \emptyset \right).
\]
\end{proposition}

\begin{proof}
The elements of $\mathcal{T}(\emptyset)$ are precisely the transversals of the collection
\[
    \left\{ \mathcal{D}(M(\mathcal{S}')) \setminus \mathcal{D}(U_{1,n}) : \mathcal{S}' \in \min(G) \setminus \{ \emptyset \} \right\}.
\]
Recall that the set of dependencies of the uniform matroid $U_{1,n}$ consists of all subsets of $[n]$ of size at least two. Hence, for each $\mathcal{S}'$, the difference $\mathcal{D}(M(\mathcal{S}')) \setminus \mathcal{D}(U_{1,n})$ corresponds exactly to the set of loops of the matroid $M(\mathcal{S}')$, which is, by definition, equal to $\mathcal{S}'$.

It follows that $\mathcal{T}(\emptyset)$ coincides with the set of transversals of the collection
\[
\{ \mathcal{S}' : \mathcal{S}' \in \min(G) \setminus \{ \emptyset \} \},
\]
which, by Lemma~\ref{transversal and connected dominant}, is precisely the set $\mathcal{D}_{c}(G)$. This proves the first part of the statement. The second part then follows directly from the definition of the ideal $J_{\mathcal{T}(\emptyset)}$.
\end{proof}

We are now prepared to establish the main result of this subsection, which offers a new proof of~\textup{\cite[Theorem~3.2]{jaramillo2024connected}}. This theorem demonstrates that the algebraic invariant $\JK$ coincides with the connected domination number $\gamma_{c}(G)$.

\begin{theorem}\label{theorem empty}
Let $G$ be a connected graph on $[n]$. Then, we have
\[\JK=\begin{cases}
\gamma_{c}(G) & \text{if $G\neq \mathcal{K}_{n},$}\\
0 & \text{if $G=\mathcal{K}_{n}$.}
\end{cases}
\]
\end{theorem}

\begin{proof}
Suppose that $G \neq \mathcal{K}_n$. By Proposition~\ref{ideal J emptyset}, we have
\[
J_{\mathcal{T}(\emptyset)} = \left( g_{C,D} : C \cup D \in \mathcal{D}_{c}(G),\ C \cap D = \emptyset \right),
\]
which gives
\[
J_G + J_{\mathcal{T}(\emptyset)} = J_G + \left( g_{C,D} : C \cup D \in \mathcal{D}_{c}(G),\ C \cap D = \emptyset \right).
\]
In~\textup{\cite[Lemma~3.5]{jaramillo2024connected}}, the authors construct a squarefree Gröbner basis for this ideal, which implies that it is radical. Hence, by Theorem~\ref{corollary radical}, it follows that
\[
\JK = \min \{\, |A_1| + 2|A_2| : A \in \mathcal{T}(\emptyset) \,\}.
\]
Moreover, Proposition~\ref{ideal J emptyset} shows that $\mathcal{T}(\emptyset) = \DG$, and we conclude that
\[
\JK = \gamma_c(G).
\]

Finally, if $G = \mathcal{K}_n$ is the complete graph, we clearly have $\JK = 0$.
\end{proof}

\section{V-numbers associated to minimal cuts}\label{section minimal cuts}

In this section, we fix a connected graph $G$ on $[n]$ along with a minimal cut $\mathcal{S}$ (see Definition~\ref{minimal cuts}). Using the reduction technique developed in Subsection~\ref{reduction with transversals}, we compute the algebraic invariant $\VPJ$. We also show that this invariant is governed by purely combinatorial properties of $G$, specifically involving connected domination numbers. Throughout this section, we fix a minimal cut $\mathcal{S} \subset [n]$ and introduce the notation that will be used in what follows.

\begin{notation}\normalfont
We introduce the following terminology:
\begin{itemize}
    \item Let $G_{1}$ and $G_{2}$ be the connected components of the induced subgraph on $[n] \setminus \mathcal{S}$, and denote by $V_{1}$ and $V_{2}$ their respective vertex sets.
\item A subset $\mathcal{S}'\subset [n]$ is said to be \emph{nice} if there do not exist vertices $i \in V_{1} \setminus \mathcal{S}'$ and $j \in V_{2} \setminus \mathcal{S}'$ such that $i$ and $j$ lie in the same connected component of $G_{[n] \setminus \mathcal{S}'}$. Observe that if $\mathcal{S}'$ is nice, we have $\mathcal{S}'\not \subset \mathcal{S}$.

\item Let $\Delta$ denote the collection
\begin{equation}\label{delta}
    \Delta := \left\{ \mathcal{S}' \setminus \mathcal{S} : \mathcal{S}' \in \min(G)\setminus \{\mathcal{S}\}, \ \text{and $\mathcal{S}'$ is nice} \right\},
\end{equation}
and let $T$ be the set of all transversals of $\Delta$. 
\end{itemize}
\end{notation}

Our first step in computing the invariant $\VPJ$ is to describe the set $\mathcal{T}(\mathcal{S})$ of transversals of the collection
\[
    \left\{ \mathcal{D}(M(\mathcal{S}')) \setminus \mathcal{D}(M(\mathcal{S})) : \mathcal{S}' \in \min(G) \setminus \{ \mathcal{S} \} \right\}.
\]
To this end, we begin by establishing a result concerning the transversal $T$, which, as we shall later see, is closely related to the set $\mathcal{T}(\mathcal{S})$. Before doing so, however, we require the following lemma concerning minimal $k$-cuts.

\begin{lemma}\label{refinement}
Let $A \subset [n]$ be a subset such that the induced subgraph $G_{[n] \setminus A}$ has connected components $G_{W_1}, \ldots, G_{W_k}$, where $k\geq 2$ and each $W_i$ denotes the vertex set of a component. Then there exists a subset $B \subset A$ such that $B \in \min(G)$, and each $W_i$ remains entirely contained within a distinct connected component of $G_{[n] \setminus B}$.
\end{lemma}

\begin{proof}
Suppose first that every vertex $i \in A$ is adjacent to at least two connected components of $G_{[n] \setminus A}$. In this case, each $i\in A$ is a cut point of the graph $G_{([n]\setminus A)\cup \{i\}}$, 
so $A \in \min(G)$, and we are done.

Otherwise, there exists a vertex $i_1 \in A$ that is adjacent to at most one connected component of $G_{[n] \setminus A}$. Define $B_1 := A \setminus \{i_1\}$. Since $i_1$ does not reconnect distinct components, each $W_i$ remains entirely contained within a distinct connected component of $G_{[n] \setminus B_1}$.

We proceed inductively: at step $k$, if there exists a vertex $i_k \in B_{k-1}$ adjacent to at most one connected component of $G_{[n] \setminus B_{k-1}}$, we set $B_k := B_{k-1} \setminus \{i_k\}$. In each step, the property that each $W_i$ lies entirely within a distinct component is preserved.

This process yields a strictly decreasing chain of subsets:
\[
A \supsetneq B_1 \supsetneq B_2 \supsetneq \cdots
\]
The process terminates by finiteness, and not with the empty set since $G$ is connected. 
Hence, there exists $\ell$ such that every vertex $i \in B_\ell$ is adjacent to at least two connected components of $G_{[n] \setminus B_\ell}$.  
Thus $B_\ell \in \min(G)$, and setting $B := B_\ell$ completes the proof.
\end{proof}

The following lemma provides a description of the set of transversals $T$, which will later be useful in computing $\mathcal{T}(\mathcal{S})$.

\begin{lemma}\label{transversal for nice subsets}
Let $B \subseteq [n] \setminus \mathcal{S}$ such that $B\cap V_{i}\neq \emptyset$ for $i=1,2$. Then $B \in T$ if and only if
\[
    B \cap V_{i} \in \mathcal{D}_{c}\big(G_{V_{i} \cup \mathcal{S}}\big) \quad \text{for each } i = 1, 2.
\]
\end{lemma}

\begin{proof}
Set $B_i := B \cap V_i$ for $i = 1,2$. Assume first that $B_i \in \mathcal{D}_{c}(G_{V_i \cup \mathcal{S}})$ for $i = 1,2$. To show that $B \in T$, it suffices to prove that
\[
B \cap \mathcal{S}' \neq \emptyset
\]
for every \emph{nice} set $\mathcal{S}' \in \min(G)$. Suppose, for the sake of contradiction, that there exists such a set $\mathcal{S}'$ with $B \cap \mathcal{S}' = \emptyset$. We distinguish two cases.

\smallskip
\noindent\textbf{Case 1: $\mathcal{S}' \supsetneq \mathcal{S}$.} Since $\mathcal{S} \in {\min}_{2}(G)$, it follows that $\mathcal{S}' \notin {\min}_{2}(G)$. Therefore, because $\mathcal{S}' \in \min(G)$, the graph $G_{[n] \setminus \mathcal{S}'}$ must have at least three connected components.

On the other hand, the assumption $B \cap \mathcal{S}' = \emptyset$ implies that $B_i \subset V_i \setminus \mathcal{S}'$ for $i = 1,2$. Since $B_1$ and $B_2$ are connected and dominant, we conclude that $G_{V_1 \setminus \mathcal{S}'}$ and $G_{V_2 \setminus \mathcal{S}'}$ are connected. Hence, $G_{[n] \setminus \mathcal{S}'}$ has at most two connected components, a contradiction.

\smallskip
\noindent\textbf{Case 2: $\mathcal{S}' \not\supset \mathcal{S}$.} Again, $B \cap \mathcal{S}' = \emptyset$ implies $B_i \subset V_i \setminus \mathcal{S}'$ for $i = 1,2$. Since $B_1$ and $B_2$ are connected and dominant in $G_{V_1 \cup \mathcal{S}}$ and $G_{V_2 \cup \mathcal{S}}$, the induced subgraphs
\[
G_{(V_1 \cup \mathcal{S}) \setminus \mathcal{S}'} \quad \text{and} \quad G_{(V_2 \cup \mathcal{S}) \setminus \mathcal{S}'}
\]
are connected.

Moreover, since $B_{1}$ and $B_{2}$ are dominant, every vertex in $\mathcal{S} \setminus \mathcal{S}'$ is adjacent to some vertex in $B_1$ and to some vertex in $B_2$. It follows that $G_{[n] \setminus \mathcal{S}'}$ is connected, contradicting the assumption that $\mathcal{S}' \in \min(G)$.

\smallskip
In both cases, we arrive at a contradiction. Hence, $B \cap \mathcal{S}' \neq \emptyset$ for all nice $\mathcal{S}' \in \min(G)$, and thus $B \in T$, as desired.

For the converse, assume that $B \in T$. We aim to show that $B_i \in \mathcal{D}_{c}(G_{V_i \cup \mathcal{S}})$ for $i = 1, 2$. Without loss of generality, we carry out the proof for $B_1$, which we divide into two steps.

\smallskip

\noindent\textbf{Step 1.} We first show that $B_1 \in \mathcal{D}_{c}(G_1)$.

\smallskip

We may assume that $G_1$ is not complete, since otherwise $B_1 \in \mathcal{D}_{c}(G_1)$ would hold automatically, given that $B_{1} \neq \emptyset$ by the assumption of the lemma. Then, by Lemma~\ref{transversal and connected dominant}, it suffices to verify that $B_1 \cap C \neq \emptyset$ for every nonempty $C \in \min(G_1)$. Assume for contradiction that $B_1 \cap C = \emptyset$ for some such $C$. Moreover, we may assume that $C$ is minimal, i.e., $C \in {\min}_2(G_1)$.

Define $A := \mathcal{S} \cup C$. Observe that $G_{[n] \setminus A}$ has exactly three connected components: $G_2$ and the two components of $G_{V_1 \setminus C}$, whose vertex sets we denote by $W_1$ and $W_2$.

By Lemma~\ref{refinement}, there exists a subset $F \subset A$ such that $F \in \min(G)$ and each of the sets $W_1$, $W_2$, and $V_2$ is entirely contained within a distinct component of $G_{[n] \setminus F}$.

\smallskip

\noindent\textbf{Claim 1.} $C \subset F$.

\smallskip

\noindent Suppose, for contradiction, that $i \in C \setminus F$ for some $i$. Since $C \in {\min}_2(G_1)$, the vertex $i$ is a cut point of $G_{(V_1 \setminus C) \cup \{i\}}$, which implies that $i$ is adjacent to some vertex in $W_{1}$ and to some vertex in $W_{2}$. Consequently, $i$ ensures that all vertices in $W_1 \cup W_2$ belong to the same connected component of $G_{[n] \setminus F}$. This contradicts the definition of $F$.

\smallskip

\noindent\textbf{Claim 2.} $F$ is nice.

\smallskip

\noindent Suppose not. Then there exist vertices $i \in V_1 \setminus F$ and $j \in V_2$ that lie in the same connected component of $G_{[n] \setminus F}$. Since $C \subset F$, we have $V_1 \setminus F = V_1 \setminus C = W_1 \cup W_2$. Without loss of generality, suppose $i \in W_1$. Then all the vertices in $W_1 \cup V_2$ lie in the same component of $G_{[n] \setminus F}$, contradicting the definition of $F$.

\smallskip

Since $F$ is nice, $F \in \min(G)$ and $B \in T$, it follows that $B \cap F \neq \emptyset$. From Claim~1, $F \setminus \mathcal{S} = C$, and since $B\cap \mathcal{S}=\emptyset$, it follows that $B \cap C \neq \emptyset$. This contradicts our assumption that $B_1 \cap C = \emptyset$. Therefore, $B_1 \in \mathcal{D}_{c}(G_1)$.

\smallskip

We have thus shown that $B_1 \in \mathcal{D}_{c}(G_1)$. It remains to show that $B_1 \in \mathcal{D}_{c}(G_{V_1 \cup \mathcal{S}})$, which by symmetry implies $B_2 \in \mathcal{D}_{c}(G_{V_2 \cup \mathcal{S}})$.

\smallskip

\noindent\textbf{Step 2.} $B_1 \in \mathcal{D}_{c}(G_{V_1 \cup \mathcal{S}})$.

\smallskip

Since we already know that $B_1 \in \mathcal{D}_{c}(G_1)$, it suffices to verify that each $x \in \mathcal{S}$ is adjacent to some vertex in $B_1$. Suppose not, and let $x \in \mathcal{S}$ be such that $\{x, b\} \notin E(G)$ for all $b \in B_1$. Define $A := [n] \setminus (B_1 \cup V_2 \cup \{x\})$.

Note that, since $x$ is not adjacent to any vertex of $B_{1}$ but is adjacent to some vertex of $V_{2}$, the sets
\[
B_1 \quad \text{and} \quad V_2\cup \{x\}
\]
are the vertex sets of the two connected components of $G_{[n] \setminus A} = G_{B_1 \cup V_2 \cup \{x\}}$. Then by Lemma~\ref{refinement}, there exists $F \subset A$ such that $F \in \min(G)$ and each of $B_{1}$ and $V_{2}\cup \{x\}$ remains entirely contained within a distinct component of $G_{[n] \setminus F}$.

\smallskip

\noindent\textbf{Claim 3.} $F$ is nice.

\smallskip
\noindent
Assume, for contradiction, that there exist vertices $r \in V_1 \setminus F$ and $s \in V_2$ that lie in the same connected component of $G_{[n] \setminus F}$. Since $B_1 \in \mathcal{D}_c(G_1)$, the vertex $r$ is adjacent to some vertex in $B_1$, which connects a vertex of $B_{1}$ with a vertex of $V_{2}\cup \{x\}$. As both $G_{B_1}$ and $G_{V_2\cup \{x\}}$ are connected subgraphs and $B_{1},V_{2}\cup \{x\}\subset [n]\setminus F$, it follows that all vertices in $B_1 \cup V_2\cup \{x\}$ belong to the same connected component of $G_{[n] \setminus F}$, contradicting the definition of $F$.

\smallskip
Thus, since $F$ is nice, $F \in \min(G)\setminus \{\mathcal{S}\}$, it follows that $F \in \Delta$, as defined in Equation~\eqref{delta}. Moreover, because $F \subset A$, we have $F \cap B = \emptyset$, contradicting the assumption that $B \in T$.

\smallskip
This contradiction shows that if $B \in T$, then each $B_i \in \mathcal{D}_c(G_{V_i \cup \mathcal{S}})$, as claimed, completing the proof.
\end{proof}

We are now in a position to describe the set of transversals $\mathcal{T}(\mathcal{S})$ of the hypergraph
\[
    \Delta_{\mathcal{S}} = \left\{ \mathcal{D}(M(\mathcal{S}')) \setminus \mathcal{D}(M(\mathcal{S})) : \mathcal{S}' \in \min(G) \setminus \{ \mathcal{S} \} \right\} 
    \subseteq \textstyle (\binom{[n]}{1} \cup \binom{[n]}{2})\setminus \mathcal{D}(M(\mathcal{S})).
\]
Observe that, by definition, every such transversal is disjoint from $\mathcal{D}(M(\mathcal{S}))$.
Before doing so, we introduce some necessary notation.

\begin{notation}
Let $A$ be a collection of subsets of $[n]$. We define the \emph{support} of $A$ as
\[
\supp(A) = \bigcup_{e \in A} e.
\]
\end{notation}

For the following proposition, recall the notions of $A_{1}$ and $A_{2}$ from Notation~\ref{important notation}.

\begin{proposition}\label{characterization TS}
A subset $A \subset \textstyle (\binom{[n]}{1} \cup \binom{[n]}{2})\setminus \mathcal{D}(M(\mathcal{S}))$ belongs to $\mathcal{T}(\mathcal{S})$ if and only if the following two conditions are satisfied:
\begin{enumerate}
\item[{\rm (i)}] $\supp(A) \cap V_i\in \mathcal{D}_c(G_{V_i \cup \mathcal{S}})$ for each $i = 1, 2$.
\item[{\rm (ii)}] $A_{2}\neq \emptyset$.
\end{enumerate}
\end{proposition}

\begin{proof}
Assume first that $A$ satisfies conditions~(i) and~(ii). We will show that $A \in \mathcal{T}(\mathcal{S})$. To that end, we must verify that $A \cap \mathcal{D}(M(\mathcal{S}')) \neq \emptyset$ for every $\mathcal{S}' \in \min(G) \setminus \{ \mathcal{S} \}$. We proceed by cases.

\smallskip
\noindent\textbf{Case~1:} $\mathcal{S}'$ is nice. By condition~(i) and Lemma~\ref{transversal for nice subsets}, we know that $\supp(A) \in T$, and hence $\supp(A) \cap \mathcal{S}' \neq \emptyset$, 
and thus there exists $r \in \supp(A) \cap \mathcal{S}'$. 

Therefore, some $e \in A$ satisfies $r \in e$. As $r\in \mathcal{S}'$, it follows that $\{r\}$ is a loop of $M(\mathcal{S}')$, so $e \in \mathcal{D}(M(\mathcal{S}'))$, establishing that $A \cap \mathcal{D}(M(\mathcal{S}')) \neq \emptyset$.

\smallskip
\noindent\textbf{Case~2:} $\mathcal{S}'$ is not nice. Then there exist $r \in V_{1} \setminus \mathcal{S}'$ and $s \in V_{2} \setminus \mathcal{S}'$ lying in the same connected component of $G_{[n] \setminus \mathcal{S}'}$. We distinguish two subcases:

\smallskip
\noindent\textbf{Case~2.1:} $\supp(A) \cap \mathcal{S}' \neq \emptyset$. Then, as in Case~1, we conclude that $A \cap \mathcal{D}(M(\mathcal{S}')) \neq \emptyset$.

\smallskip
\noindent\textbf{Case~2.2:} $\supp(A) \cap \mathcal{S}' = \emptyset$. Then $\supp(A) \subseteq [n] \setminus \mathcal{S}'$. By condition~(i), the induced subgraphs $G_{V_1 \setminus \mathcal{S}'}$ and $G_{V_2 \setminus \mathcal{S}'}$ are connected. Since $r$ and $s$ lie in the same connected component of $G_{[n] \setminus \mathcal{S}'}$, it follows that all vertices in $(V_1 \cup V_2) \setminus \mathcal{S}'$ are in a single connected component of $G_{[n] \setminus \mathcal{S}'}$. Consequently, every $2$-subset of $V_1 \cup V_2$ is dependent in $M(\mathcal{S}')$. By condition~(ii) it follows that there exist vertices $i \in V_1$ and $j \in V_2$ such that $\{i, j\} \in A$.
Consequently, since $\{i,j\}\in \mathcal{D}(M(\mathcal{S}'))$, it follows that $A \cap \mathcal{D}(M(\mathcal{S}')) \neq \emptyset$.

\smallskip
This proves that $A \in \mathcal{T}(\mathcal{S})$.

\smallskip

For the converse, suppose that $A \in \mathcal{T}(\mathcal{S})$. We will show that $A$ satisfies conditions~(i) and~(ii).

We first verify condition~(ii). Since $A \in \mathcal{T}(\mathcal{S})$ and $\emptyset \in \min(G)$, it follows that
\[
A \cap  \mathcal{D}(M(\emptyset))  \neq \emptyset.
\]
Note that $M(\emptyset)$ is the rank-one uniform matroid $U_{1,n}$, so
\[
\mathcal{D}(M(\emptyset)) = \{ B \subseteq [n] : |B| \geq 2 \}.
\]
On the other hand,
\[
\mathcal{D}(M(\mathcal{S})) = \{ B \subseteq [n] : B \cap \mathcal{S} \neq \emptyset \} \cup \{ B : |B \cap V_1| \geq 2 \text{ or } |B \cap V_2| \geq 2 \} \cup \{ B : |B| \geq 3 \}.
\]
Therefore,
\[
\mathcal{D}(M(\emptyset)) \setminus \mathcal{D}(M(\mathcal{S})) = \{ B \subseteq [n] : |B| = 2 \text{ and } |B \cap V_1| = |B \cap V_2| = 1 \}.
\]
It follows that $A$ contains a subset of the form $\{i, j\}$ with $i \in V_1$ and $j \in V_2$, establishing condition~(ii).

By condition~(ii), it follows that $\supp(A)\cap V_{i}\neq \emptyset$ for $i=1,2$. Then, by Lemma~\ref{transversal for nice subsets}, condition~(i) is equivalent to $\supp(A) \in T$. We argue by contradiction and assume that $\supp(A) \notin T$. Then there exists a nice subset $\mathcal{S}' \in \min(G)\setminus \{\mathcal{S}\}$ such that $\supp(A) \cap \mathcal{S}' = \emptyset$. Since $\mathcal{S}'$ is nice, every element in
\[
\mathcal{D}(M(\mathcal{S}')) \setminus \mathcal{D}(M(\mathcal{S}))
\]
is a subset of size one or two and contains at least one element of $\mathcal{S}'$. But as $\supp(A) \cap \mathcal{S}' = \emptyset$, it follows that $A$ cannot contain an element of this form, i.e., 
\[
A \cap \mathcal{D}(M(\mathcal{S}'))  = \emptyset.
\]
This contradicts the assumption that $A \in \mathcal{T}(\mathcal{S})$. Hence, $\supp(A) \in T$, and condition~(i) holds. This completes the proof.
\end{proof}

Proposition~\ref{characterization TS} motivates the introduction of a new graph-theoretic notion of dominance, adapted to the context of minimal cuts, which we now define. 

\begin{definition}\label{definition gACDij 21}
Define
\[
\mathcal{D}_{c}(V_{1}, V_{2}) = \left\{ A \subset V_{1} \cup V_{2} \mid A \cap V_{i} \in \mathcal{D}_{c}(G_{V_{i} \cup \mathcal{S}})\ \text{for each } i = 1, 2 \right\}.
\]
Moreover, we define the number
\[\gamma_{c}(V_{1},V_{2})=\min\{\size{A}:A\in \DVV\}.\]
Observe that $V_{1}\cup V_{2}\in \mathcal{D}_{c}(V_{1}, V_{2})$, which implies that $\gamma_{c}(V_{1},V_{2})$ is well-defined.
\end{definition}

From the description of $\mathcal{T}(\mathcal{S})$ given in Proposition~\ref{characterization TS}, one obtains an explicit finite generating set for the ideal $\JS$, as specified in Definition~\ref{notation JS}. Our aim, however, is to provide a smaller generating set. These generators are introduced in the following definition.

\begin{definition}\label{definition gACDij}
For any $A \in \mathcal{D}_{c}(V_{1}, V_{2})$, elements $i \in A \cap V_{1}$ and $j \in A \cap V_{2}$, and a partition $\{C, D\}$ of $A \setminus \{i, j\}$, we define
\[
g_{A,C,D}^{i,j} := g_{C,D} \cdot f_{i,j}=g_{C,D} \cdot (x_{i}y_{j}-x_{j}y_{i}).
\]
\end{definition}

We are now in a position to present a concise generating set for the ideal $\JS$.

\begin{lemma}\label{concise generating set}
The ideal $\JS$ admits the following generating set:
\begin{equation}\label{equality JS}
    \JS = \left( g_{A,C,D}^{i,j} \ : \ A \in \mathcal{D}_{c}(V_1, V_2),\ i \in A \cap V_1,\ j \in A \cap V_2,\ \text{and}\ C \amalg D = A \setminus \{i,j\} \right).
\end{equation}
\end{lemma}

\begin{proof}
Let $J$ denote the ideal appearing on the right-hand side of Equation~\eqref{equality JS}. We begin by proving the inclusion $J \subset \JS$. To this end, let $g_{A,C,D}^{i,j}$ be an arbitrary generator of $J$, and we will show that it lies in $\JS$.

Consider the collection of subsets
\[
A' := \{\{x\} : x \in A \setminus \{i,j\}\} \cup \{\{i,j\}\}.
\]
Note that $\supp(A') = A \in \DVV$, and $A'$ contains the pair $\{i,j\}$, with $i \in V_1$ and $j \in V_2$. Therefore, by Proposition~\ref{characterization TS}, it follows that $A' \in \TS$.

Now, using the notation introduced in Definition~\ref{important notation} (iii), we have
\[
A'_1 = \{\{x\} : x \in A \setminus \{i,j\}\} \quad \text{and} \quad A'_2 = \{\{i,j\}\},
\]
and then
\[
g_{A,C,D}^{i,j} = g_{A',C,D}.
\]
Since $g_{A',C,D} \in \JS$ by Definition~\ref{notation JS}, we conclude that $g_{A,C,D}^{i,j} \in \JS$, as required.


We now turn to the proof of the reverse inclusion. Let $g_{B,C,D}$ be an arbitrary generator of $\JS$. By Definition~\ref{notation JS}, this generator corresponds to a subset $B \in \TS$ and a partition $C \amalg D = B_1$, and our goal is to show that $g_{B,C,D} \in J$. We proceed by induction on the quantity $|B_2|$, the number of subsets of size two in $B$.

\smallskip
\noindent\textbf{Base case: $\boldsymbol{|B_2| = 1}$.} Suppose $B_2 = \{\{r,s\}\}$. Then, by definition,
\[
g_{B,C,D} = g_{C,D} \cdot f_{r,s}.
\]
Since $B \in \TS$, Proposition~\ref{characterization TS} implies that
\[
C \cup D \cup \{r,s\} \in \DVV.
\]
Consider the following generator of $J$:
\begin{equation}\label{long polynomial}
g^{r,s}_{C \cup D \cup \{r,s\},\, C \setminus \{r,s\},\, D \setminus \{r,s\}} = g_{C \setminus \{r,s\},\, D \setminus \{r,s\}} \cdot f_{r,s}.
\end{equation}
By construction, the polynomial in~\eqref{long polynomial} belongs to $J$, and clearly divides $g_{B,C,D}$. It follows that $g_{B,C,D} \in J$, as claimed.

\smallskip

\noindent \textbf{Inductive step.} Let $B \in \mathcal{T}(\mathcal{S})$ with $|B_2| \geq 2$, and assume the statement holds for all $U \in \mathcal{T}(\mathcal{S})$ such that $|U_2| < |B_2|$. Let $\{r, s\}$ be an element of $B_2$ and consider the collection of subsets
\[
B' := B \setminus \{\{r,s\}\}.
\]
We look at the possible cases for the quantity $\size{\{r,s\}\cap \supp(B')}$.

\smallskip 

{\textbf{Case 1:} $\{r,s\} \subset \supp(B')$.} In this case, note that $\supp(B') = \supp(B)$ and that $B'_2 \neq \emptyset$. By Proposition~\ref{characterization TS}, it follows that $B' \in \TS$. Moreover, since $\size{B'_2} = \size{B_2} - 1$, the inductive hypothesis ensures that
\[
g_{B', C, D} \in J
\]
Since $g_{B,C,D}=g_{B', C, D}\cdot f_{r,s}$, it follows that $g_{B,C,D}\in J$. 

\smallskip

{\textbf{Case 2:} $\{r,s\} \cap \supp(B') = \emptyset$.} 
In this case, consider the collection of subsets
\[
B'' := (B \cup \{\{r\}, \{s\}\}) \setminus \{\{r,s\}\}.
\]
Note that $\supp(B'') = \supp(B)$ and that $B''_2 \neq \emptyset$. By Proposition~\ref{characterization TS}, it follows that $B'' \in \TS$. Moreover, since $\size{B''_2} = \size{B_2} - 1$, the inductive hypothesis ensures that
\[
g_{B'', C', D'} \in J
\]
for any partition $\{C', D'\}$ of $B_1 \cup \{r,s\}$. In particular, this applies to the polynomials
\[
g_{B'', C \cup \{r\}, D \cup \{s\}} \quad \text{and} \quad 
g_{B'', C \cup \{s\}, D \cup \{r\}},
\]
which therefore lie in $J$. By the definition of $B''$, we also have the identities
\[
g_{B'', C \cup \{r\}, D \cup \{s\}} = \frac{g_{B,C,D} \, x_r y_s}{f_{r,s}} 
\quad \text{and} \quad 
g_{B'', C \cup \{s\}, D \cup \{r\}} = \frac{g_{B,C,D} \, x_s y_r}{f_{r,s}}.
\]
Subtracting these two expressions yields
\[
g_{B,C,D} = \frac{g_{B,C,D} \, x_r y_s}{f_{r,s}} - \frac{g_{B,C,D} \, x_s y_r}{f_{r,s}} \in J,
\]
as desired.

\smallskip

{\bf Case~3: $\lvert\{r,s\}\cap \supp(B')\rvert=1$.} 
Assume, without loss of generality, that $r\in \supp(B')$ and $s\notin \supp(B')$. In this case, consider the collection of subsets
\[B'':=(B\cup \{\{s\}\})\setminus \{\{r,s\}\}.\]
Note that $\supp(B'')=\supp(B)$ and that $B''_{2}\neq \emptyset$. By Proposition~\ref{characterization TS}, it follows that $B''\in \TS$. Moreover, since $\size{B''_{2}}=\size{B_{2}}-1$, the inductive hypothesis ensures that
\[g_{B'',C',D'}\in J\]
for any partition $\{C', D'\}$ of $B_1 \cup \{s\}$. In particular, this applies to the polynomials
\[
g_{B'', C \cup \{s\}, D} \quad \text{and} \quad 
g_{B'', C, D \cup \{s\}},
\]
which therefore lie in $J$. By the definition of $B''$, we also have the identities
\[
g_{B'', C \cup \{s\}, D} = \frac{g_{B,C,D} \, x_s}{f_{r,s}} 
\quad \text{and} \quad 
g_{B'', C , D \cup \{s\}} = \frac{g_{B,C,D} \, y_s}{f_{r,s}}.
\]
Consequently, we have
\[
g_{B,C,D} = \frac{g_{B,C,D} \, y_s}{f_{r,s}}\cdot x_{r} - \frac{g_{B,C,D} \, x_s}{f_{r,s}}\cdot y_{r} \in J,
\]
as desired. This completes the proof.
\end{proof}

\begin{definition}\label{def concise gen set}
For ease of notation, we denote by $\mathcal{G}_{\mathcal{S}}$ the generating set of the ideal $\JS$ described in Lemma~\ref{concise generating set}.
\end{definition}

Having established a concise generating set for the ideal $\mathcal{J}_{\mathcal{S}}$, the next natural step, following the approach outlined in Subsection~\ref{reduction with transversals}, is to determine the radical of $J_G + \JS$. To this end, we will compute a Gröbner basis for the ideal. Before proceeding, we recall some relevant definitions and results from~\cite{herzog2010binomial}.

\begin{definition}[\cite{herzog2010binomial}]\label{lexicographical order}
We adopt the following terminology:

\begin{itemize}
    \item Let $<$ denote the lexicographic order on the polynomial ring $\CC[x_1,\ldots,x_n,y_1,\ldots,y_n]$ induced by
    \[
    x_1 > x_2 > \cdots > x_n > y_1 > y_2 > \cdots > y_n.
    \]

    \item Let $G$ be a graph on the vertex set $[n]$, and let $i,j$ be vertices of $G$ with $i < j$. A path 
    \[
    \pi: i = i_0, i_1, \ldots, i_r = j
    \]
    from $i$ to $j$ is called \emph{admissible} if the following conditions hold:
    \begin{itemize}
        \item[{\rm (i)}] The vertices $i_0, \ldots, i_r$ are pairwise distinct.
        \item[{\rm (ii)}] For each $1 \leq k \leq r-1$, either $i_k < i$ or $i_k > j$.
        \item[{\rm (iii)}] For any proper subset $\{j_1, \ldots, j_s\}$ of $\{i_1, \ldots, i_{r-1}\}$, the sequence $i, j_1, \ldots, j_s, j$ is not a path in $G$.
    \end{itemize}
    To an admissible path $\pi$, we associate the monomial
    \[
    u_\pi = \prod_{i_k > j} x_{i_k} \cdot \prod_{i_l < i} y_{i_l}.
    \]

    \item Let $i$ and $j$ be vertices of a graph $G$, and let $B \subseteq V(G)$. We say that $i$ and $j$ are \emph{connected through} $B$ if there exists a path 
    \[
    i, j_1, \ldots, j_s, j
    \]
    from $i$ to $j$ such that $\{j_1, \ldots, j_s\} \subseteq B$.
\end{itemize}
\end{definition}

The following proposition is essential for the computation of a Gröbner basis for the ideal $J_G + \JS$.

\begin{proposition}\label{grobner basis for JG}
Let $G$ be a graph on $[n]$, and let $<$ denote the monomial order defined in Definition~\ref{lexicographical order}. Then the following hold:
\begin{enumerate}
\item[{\rm (i)}] The set of binomials
\[
 \bigcup_{i<j} \left\{ u_{\pi} f_{i,j} \,\middle|\, \pi \text{ is an admissible path from } i \text{ to } j \right\}
\]
forms a reduced Gröbner basis of the ideal $J_G$.
\item[{\rm (ii)}] If $i$ and $j$ are vertices connected through a subset $B \subseteq V(G)$, and $\{C, D\}$ is a partition of $B$, then the polynomial
\[
g_{C,D} \cdot f_{i,j}
\]
belongs to $J_G$.
\end{enumerate}
\end{proposition}

\begin{proof}
(i) The proof of this result is given in \textup{\cite[Theorem~2.1]{herzog2010binomial}}.

\smallskip

(ii) We begin by showing that $g_{C,D} f_{i,j}$ vanishes at every point $\gamma \in \mathbb{V}(J_G)$. Let 
\[
\gamma = \{\gamma_1, \ldots, \gamma_n\} \in \mathbb{V}(J_G)
\]
be an arbitrary element of the variety, where each $\gamma_k$ is a vector in $\CC^2$ indexed by $[n]$.

Since $i$ and $j$ are connected through $B$, there exists a path $i, j_1, \ldots, j_s, j$ in $G$ such that $\{j_1, \ldots, j_s\} \subset B$. As $\gamma \in \mathbb{V}(J_G)$, and since each consecutive pair in this path corresponds to an edge of $G$, it follows that each of the following pairs of vectors
\begin{equation}\label{ld}
\{\gamma_i, \gamma_{j_1}\}, \{\gamma_{j_1}, \gamma_{j_2}\}, \ldots, \{\gamma_{j_{s-1}}, \gamma_{j_s}\}, \{\gamma_{j_s}, \gamma_j\}
\end{equation}
is linearly dependent. We now consider two cases:

\smallskip
\noindent
\textbf{Case 1.} Suppose there exists $k \in [s]$ such that $\gamma_{j_k} = 0$. Since $j_k \in B$ and, by construction, $g_{C,D}$ is divisible by either $x_{j_k}$ or $y_{j_k}$, it follows that $g_{C,D}(\gamma) = 0$.

\smallskip
\noindent
\textbf{Case 2.} Suppose $\gamma_{j_k} \neq 0$ for all $k \in [s]$. Then, by Equation~\eqref{ld}, the vectors $\gamma_{j_1}, \ldots, \gamma_{j_s}$ must all be scalar multiples of one another. Moreover, since both $\{\gamma_i, \gamma_{j_1}\}$ and $\{\gamma_{j_s}, \gamma_j\}$ are linearly dependent, it follows that $\gamma_i$ and $\gamma_j$ are linearly dependent. Hence, $f_{i,j}(\gamma) = 0$.

\smallskip
In either case, we conclude that $g_{C,D} f_{i,j}(\gamma) = 0$ for all $\gamma \in \mathbb{V}(J_G)$. Therefore,
\[
g_{C,D} f_{i,j} \in \mathbb{I}(\mathbb{V}(J_G)) = \sqrt{J_G}.
\]
Finally, since $J_G$ is radical by \textup{\cite[Corollary~2.2]{herzog2010binomial}}, we conclude that $g_{C,D} f_{i,j} \in J_G$. \qedhere
\end{proof}

We are now in a position to describe a Gröbner basis for the ideal $J_G + \JS$. Before doing so, we introduce a simplifying assumption.

\begin{assumption}\label{assumption 2}
In the proposition below, we assume that 
if $i \in V_1$ and $i' \notin V_1$, then $i < i'$, and if $j \in V_2$ and $j' \notin V_2$, then $j' < j$.
\end{assumption}

This assumption is made without loss of generality, as such an ordering can always be achieved by relabeling the vertex set.

\begin{proposition}\label{proposition grobner basis JG+JS}
With the notation of Definitions~\ref{def concise gen set} and~\ref{lexicographical order}, the set
\[
\mathcal{G} = \bigcup_{i<j} \left\{ u_{\pi} f_{i,j} \,\middle|\, \pi \text{ is an admissible path from } i \text{ to } j \right\} \cup \GS
\]
forms a Gröbner basis for the ideal $J_G + \JS$ with respect to the lexicographical order.
\end{proposition}

\begin{proof}
Set 
\[
\mathcal{G}' = \bigcup_{i < j} \left\{ u_{\pi} f_{i,j} \,\middle|\, \pi \text{ is an admissible path from } i \text{ to } j \right\}.
\]
By Lemma~\ref{concise generating set} and Proposition~\ref{grobner basis for JG}~(i), the union $\mathcal{G}' \cup \GS$ generates the ideal $J_G + \JS$. To show that $\mathcal{G}$ is a Gröbner basis, we use Buchberger’s criterion and verify that every $S$-pair $S(f,h)$ with $f,h \in \mathcal{G}$ reduces to zero. We proceed by analyzing the different possible cases for the pair $(f,h)$.

\smallskip
\noindent\textbf{Case~1:} $f,h \in \mathcal{G}'$.  
By Proposition~\ref{grobner basis for JG}~(i), the set $\mathcal{G}'$ is a Gröbner basis for $J_G$. It follows that $S(f,h)$ reduces to zero.

\smallskip
\noindent\textbf{Case~2:} $f,h \in \GS$.  
Write $f = g_{A,C,D}^{i,j}$ and $h = g_{A',C',D'}^{k,l}$, where $\{i,k\} \subset V_1$ and $\{j,l\} \subset V_2$. We distinguish cases based on the indices.

\smallskip
\noindent\textbf{Case~2.1:} $i = k$ and $j = l$.  
In this situation, one has $S(f,h) = 0$.

\smallskip
\noindent\textbf{Case~2.2:} $i \neq k$ and $j \neq l$.  
The initial monomials $\textup{in}_{<}(f_{i,j})$ and $\textup{in}_{<}(f_{k,l})$ form a regular sequence. As a result, the $S$-pair $S(f,h)$ reduces to zero, because of the following more general fact: if $f_1, f_2$ are polynomials whose initial monomials form a regular sequence, then for any monomials $u,v$, the $S$-pair $S(uf_1, vf_2)$ reduces to zero, see \textup{\cite[Problem~2.17]{ene2011grobner}}.

\smallskip
\noindent\textbf{Case~2.3:} $\lvert \{i,j\} \cap \{k,l\} \rvert = 1$.  
In this case, either $i = k$ and $j \neq l$, or $j = l$ and $i \neq k$. Without loss of generality, assume $i = k$ and $j \neq l$, as the other case is analogous. Then:
\begin{align*}
S(f,h) 
&= \frac{\mathrm{lcm}(g_{C,D}x_i y_j,\, g_{C',D'}x_i y_l)}{x_i y_j}(x_i y_j - x_j y_i) 
  - \frac{\mathrm{lcm}(g_{C,D}x_i y_j,\, g_{C',D'}x_i y_l)}{x_i y_l}(x_i y_l - x_l y_i) \\
&= \mathrm{lcm}(g_{C,D}x_i y_j,\, g_{C',D'}x_i y_l) \left( \frac{x_l y_i}{x_i y_l} - \frac{x_j y_i}{x_i y_j} \right) \\
&= \frac{\mathrm{lcm}(g_{C,D}x_i y_j,\, g_{C',D'}x_i y_l)}{x_i y_j y_l} \cdot y_i f_{l,j}.
\end{align*}

We now verify that the expression above lies in $J_G$. Note that
\[
g_{C \setminus \{i,j,l\},\, D \cup \{i\} \setminus \{j,l\}} \;\Big| \; \frac{\mathrm{lcm}(g_{C,D}x_i y_j,\, g_{C',D'}x_i y_l) \cdot y_i}{x_i y_j y_l},
\]
and therefore,
\begin{equation}\label{divides}
g_{C \setminus \{i,j,l\},\, D \cup \{i\} \setminus \{j,l\}} \cdot f_{l,j} \;\Big| \; S(f,h).
\end{equation}

Moreover, since $g_{A,C,D}^{i,j} \in \GS$, we have that $C \cup D \cup \{i,j\} \in \DVV$, implying that this set contains a connected dominating set of $V_2$. In particular, as $\{j, l\} \subset V_2$, there exists a path from $j$ to $l$ within $(C \cup D \cup \{i\}) \setminus \{j,l\}$. Observe that
\[
(C \setminus \{i,j,l\}) \cup (D \cup \{i\}) = (C \cup D \cup \{i\}) \setminus \{j,l\},
\]
so by Proposition~\ref{grobner basis for JG}~(ii),
\[
g_{C \setminus \{i,j,l\},\, D \cup \{i\} \setminus \{j,l\}} \cdot f_{l,j} \in J_G.
\]
Then, combining this with~\eqref{divides}, we conclude that $S(f,h) \in J_G$. Since $\mathcal{G}'$ is a Gröbner basis for $J_G$ by Proposition~\ref{grobner basis for JG}~(i), it follows that $S(f,h)$ reduces to zero.

\smallskip
\noindent\textbf{Case~3:} $f \in \GS$ and $h \in \mathcal{G}'$.  
Let $f = g_{A,C,D}^{i,j}$ and $h = u_{\pi} f_{k,l}$, where $i \in V_1$, $j \in V_2$, $k < l$, and $\pi$ is an admissible path from $k$ to $l$.

If $\{i,j\} = \{k,l\}$, then clearly $S(f,h) = 0$.  
If $\{i,j\}\cap \{k,l\}=\emptyset$, it follows using the same argument as in Case~2.2 that $S(f,h)$ reduces to zero.

It remains to consider the case $\lvert \{i,j\} \cap \{k,l\} \rvert = 1$. Without loss of generality, we may assume that $\{i,j\} \cap \{k,l\} = \{i\}$, since the case $\{i,j\} \cap \{k,l\} = \{j\}$ is analogous. We distinguish two subcases:

\smallskip
{\bf Case~3.1: $i=l$.} Since we are working under the assumption~\ref{assumption 2}, we have that $i<j$. In this case, we have
\[\textup{in}_{<}(f_{i,j})=x_{i}y_{j}\quad \text{and} \quad \textup{in}_{<}(f_{k,i})=x_{k}y_{i}.\]
Therefore, in this case $\textup{in}_{<}(f_{i,j})$ and $\textup{in}_{<}(f_{k,l})$ form a regular sequence. Consequently, arguing as in Case~2.2, we conclude that $S(f,h)$ reduces to zero. 

\smallskip
\noindent\textbf{Case~3.2:} $i = k$. In this case, we compute:
\begin{equation*}
\begin{aligned}
S(f,h) &= \frac{\textup{lcm}(g_{C,D} x_i y_j,\, g_{C',D'} x_i y_l)}{x_i y_j}(x_i y_j - x_j y_i)
- \frac{\textup{lcm}(g_{C,D} x_i y_j,\, g_{C',D'} x_i y_l)}{x_i y_l}(x_i y_l - x_l y_i) \\
&= \textup{lcm}(g_{C,D} x_i y_j,\, g_{C',D'} x_i y_l) \left( \frac{x_l y_i}{x_i y_l} - \frac{x_j y_i}{x_i y_j} \right) \\
&= \frac{\textup{lcm}(g_{C,D} x_i y_j,\, g_{C',D'} x_i y_l)}{x_i y_j y_l} \, y_i f_{l,j}.
\end{aligned}
\end{equation*}
We now distinguish two subcases, according to whether $l \in V_1$ or $l \notin V_1$.

\smallskip
\noindent\textbf{Case~3.2.1:} $l \notin V_{1}$. In this case, we have $l \in \mathcal{S} \cup V_{2}$. Since $g_{A,C,D}^{i,j} \in \GS$, it follows that $C \cup D \cup \{i,j\} \in \DVV$, which in particular implies that this set contains a connected dominating set of $\mathcal{S} \cup V_{2}$. In particular, since $\{j, l\} \subset \mathcal{S} \cup V_{2}$, we deduce that $j$ and $l$ are connected through $(C \cup D \cup \{i\}) \setminus \{j, l\}$, and since
\[
(C \setminus \{i,j,l\}) \cup (D \cup \{i\}) = (C \cup D \cup \{i\}) \setminus \{j, l\},
\]
it follows from Proposition~\ref{grobner basis for JG}~(ii) that
\[
g_{C \setminus \{i,j,l\},\, D \cup \{i\} \setminus \{j,l\}} f_{l,j} \in J_G.
\]
As in Case~2.3, we also have that
\[
g_{C \setminus \{i,j,l\},\, D \cup \{i\} \setminus \{j,l\}} f_{l,j} \mid S(f,h),
\]
from which it follows that $S(f,h) \in J_G$. Since $\mathcal{G}'$ is a Gröbner basis for $J_G$ (by Proposition~\ref{grobner basis for JG}~(i)), we conclude that $S(f,h)$ reduces to zero.

\smallskip
\noindent\textbf{Case~3.2.2:} $l \in V_{1}$. Note that
\[
C \setminus \{i,j,l\} \quad \text{and} \quad D \cup \{i\} \setminus \{j,l\}
\]
form a partition of $(C \cup D \cup \{i\}) \setminus \{j,l\}$. Since $C \cup D \cup \{i,j\} \in \DVV$, and $l \in V_{1}$, $j \in V_{2}$, it follows from Definition~\ref{definition gACDij} that
\[
g_{C \setminus \{i,j,l\},\, D \cup \{i\} \setminus \{j,l\}} f_{l,j} = g^{l,j}_{C \cup D \cup \{i,j,l\},\, C \setminus \{i,j,l\},\, D \cup \{i\} \setminus \{j,l\}} \in \GS.
\]
Moreover, as in the previous case, this polynomial divides $S(f,h)$. Hence, $S(f,h)$ is divisible by an element of $\GS$, and therefore reduces to zero.

\smallskip

Having treated all possible cases and shown that every $S$-pair $S(f,h)$ reduces to zero, we conclude that $\mathcal{G}$ is a Gröbner basis for $J_{G} + \JS$.
\end{proof}

As a consequence of Proposition~\ref{proposition grobner basis JG+JS}, it follows that the ideal $J_G + \JS$ is radical.

\begin{corollary}\label{is radical}
$J_G + \JS$ is a radical ideal.
\end{corollary}

\begin{proof}
The assertion follows from Proposition~\ref{proposition grobner basis JG+JS} and the following general fact: let $I\subset S$
be a graded ideal with the property that $\textup{in}_{<}(I)$ is squarefree for some monomial order $<$. Then $I$ is a radical ideal.
\end{proof}

We are now ready to state the main result of this section. This theorem shows that the algebraic invariant $\VPJ$ agrees with the number $\gamma_{c}(V_{1},V_{2})$ introduced in Definition~\ref{definition gACDij 21},  and provides a combinatorial interpretation of this invariant in terms of our new graph-theoretic notion of dominance specifically adapted to the setting of minimal cuts.

\begin{theorem}\label{main theorem}
Let $G$ be a connected graph on $[n]$ and let $\mathcal{S}$ be a minimal cut of $G$. Then, we have
\[
\VPJ = \gamma_{c}(V_{1}, V_{2}) = \min\left\{ \lvert A \rvert : A \in \DVV \right\},
\]
where
\[
\DVV = \left\{ A \subset V_{1} \cup V_{2} \mid A \cap V_{i} \in \mathcal{D}_{c}(G_{V_{i} \cup \mathcal{S}})\ \text{for each } i = 1, 2 \right\}.
\]
\end{theorem}

\begin{proof}
By Corollary~\ref{is radical}, the ideal $J_{G} + \JS$ is radical. Therefore, Theorem~\ref{corollary radical} implies that
\[
\VPJ = \min \left\{\, |A_1| + 2|A_2| : A \in \mathcal{T}(\mathcal{S}) \,\right\}.
\]
Moreover, Proposition~\ref{characterization TS} characterizes the sets $A \in \mathcal{T}(\mathcal{S})$ as those for which $\supp(A) \in \DVV$ and there exist vertices $i \in V_1$ and $j \in V_2$ such that $\{i, j\} \in A$. We thus conclude that
\[
\VPJ = \min\left\{ \lvert A \rvert : A \in \DVV \right\} = \gamma_c(V_1, V_2).
\]
\end{proof}

\section{V-number of cycle graphs} \label{section v numbers of cycle graphs}

In this section, we study the v-number of the cycle graph $C_n$, which we now define.

\begin{definition}\label{definition cycle graph}
The cycle graph $C_{n}$ is defined by
\[
V(C_n) = [n] \quad \text{and} \quad E(C_n) = \{(i, i+1) : i \in [n]\},
\]
with the convention that $n+1 = 1$.
\end{definition}

It was established in \textup{\cite[Corollary~4.10]{dey2024v}} that
\begin{equation}\label{conjecture}
\textup{v}(J_{C_n}) \leq \textstyle \left\lceil \frac{2n}{3} \right\rceil,
\end{equation}
and it was conjectured that equality holds for $n\geq 6$. Our main result in this section, Theorem~\ref{main theorem 2}, provides an almost complete confirmation of this conjecture: we show that
\[
\VCNN \in \textstyle \left\{ \left\lceil \frac{2n}{3} \right\rceil,\ \left\lceil \frac{2n}{3} \right\rceil - 1 \right\}.
\]

\begin{remark}
The set $\min(C_n)$ consists of the empty set together with all subsets $\mathcal{S} \subset [n]$ of size at least two for which no two elements are adjacent in $C_{n}$; that is, $(i,j) \notin E(C_{n})$ for all $i, j \in \mathcal{S}$.
\end{remark}

To compute $\VCNN$, we fix a subset $\mathcal{S} \in \min(C_n)$ with $\size{\mathcal{S}} \geq 3$, and study the localized v-number $\textup{v}_{P_{\mathcal{S}}}(J_{C_n})$. 

\begin{notation}\label{notation c1 and c2}
We establish the following terminology:
\begin{itemize}
\item A subset of $[n]$ is called a \emph{interval} if it is of the form $[i,j]$ or $[1,i]\cup [j,n]$ for some $1\leq i\leq j\leq n$. 
For $1\leq i\leq j\leq n$, we denote the subset $[1,i]\cup [j,n]$ by $[j,i]$.
\item The complement of $\mathcal{S}$ is a disjoint union of intervals $I_1, \ldots, I_k$, where each interval is denoted by $I_j = [a_j, b_j]$ and $a_{j+1}=b_{j}+2$ for each $j \in [k]$. Note that $k=\size{\mathcal{S}}$ and \[\{b_{1}+1,\ldots,b_{k}+1\}=\{a_{1}-1,\ldots,a_{k}-1\}=\mathcal{S}.\] 
\item We set
\[
C_{1}(\mathcal{S}) = \{ j \in [k] \,:\, |I_j| = 1 \}  
\quad \text{and} \quad 
C_{2}(\mathcal{S}) = \{ j \in [k] \,:\, |I_j| \geq 2 \}  .
\]
\end{itemize}
\end{notation}

Our first objective is to address the following problem:

\begin{question}\label{localized for cycles}
Determine the localized v-number $\textup{v}_{P_{\mathcal{S}}}(J_{C_n})$.
\end{question}

Following the approach outlined is Subsection~\ref{reduction with transversals}, we start by giving a description of the set of transversals $\TS$.

\begin{lemma}\label{charact TS}
A subset $A \subset \textstyle (\binom{[n]}{1} \cup \binom{[n]}{2})\setminus \mathcal{D}(M(\mathcal{S}))$ belongs to $\mathcal{T}(\mathcal{S})$ if and only if the following two conditions are satisfied:
\begin{itemize}
\item[{\rm (i)}] $\supp(A)=[n]\setminus \mathcal{S}$.
\item[{\rm (ii)}] For each $j \in [k]$, there exist vertices $u \in I_j$ and $w \in I_{j+1}$ such that $\{u, w\} \in A$.
\end{itemize}
\end{lemma}

\begin{proof}
Assume first that $A \in \TS$, and we will prove that conditions (i) and (ii) are necessary. We begin with condition (ii).

Fix $j \in [k]$ and consider the element $b_j + 1 \in \mathcal{S}$. Since $\mathcal{S} \setminus \{b_j + 1\} \in \min(C_n)$, and $A \in \TS$, it follows that
\[
A \cap \mathcal{D}(M(\mathcal{S} \setminus \{b_j + 1\}))  \neq \emptyset.
\]
Observe that
\[
\mathcal{D}(M(\mathcal{S} \setminus \{b_j + 1\})) \setminus \DMS = \{ \{u,v\} : u \in I_j, v \in I_{j+1} \}.
\]
Hence, given that $A\cap \DMS=\emptyset$, it follows that for each $j \in [k]$, the set $A$ contains a pair $\{u,v\}$ with $u \in I_j$ and $v \in I_{j+1}$, which verifies condition (ii).

We now verify condition (i). Let $x \in [n] \setminus \mathcal{S}$ and suppose $x \in I_j = [a_j, b_j]$. We claim that $x \in \supp(A)$. We consider cases:

\smallskip
\noindent\textbf{Case 1:} $x \notin \{a_j, b_j\}$. In this case, let $\mathcal{S}' = \mathcal{S} \cup \{x\}$. Since $x \notin \{a_j, b_j\}$, it follows that $\mathcal{S}' \in \min(C_n)$. Since $A \in \TS$, we have
\[
A \cap \left( \mathcal{D}(M(\mathcal{S}')) \setminus \DMS \right) \neq \emptyset.
\]
Moreover, every element of $\mathcal{D}(M(\mathcal{S}')) \setminus \DMS$ contains $x$, so $x \in \supp(A)$.

\smallskip
\noindent\textbf{Case 2:} $x = a_j$ or $x = b_j$. Without loss of generality, assume $x = a_j$ (the other case is analogous). If $|I_j| = 1$, then $I_j = \{x\}$. Condition (ii) implies that $A$ contains a pair $\{u,v\}$ with $u \in I_j$, and hence $u = x$, so $x \in \supp(A)$.

Now suppose $|I_j| \geq 2$. Define $\mathcal{S}' = (\mathcal{S} \cup \{x\}) \setminus \{a_j - 1\}$. Since $|I_j| \geq 2$, it follows that $\mathcal{S}' \in \min(C_n)$. Then, since $A \in \TS$, we again obtain
\begin{equation} \label{nonempty}
A \cap \left( \mathcal{D}(M(\mathcal{S}')) \setminus \DMS \right) \neq \emptyset.
\end{equation}
As before, each subset in the difference contains $x$, so $x \in \supp(A)$.

\smallskip
This shows that $[n] \setminus \mathcal{S} \subset \supp(A)$. Moreover, since $A \cap \DMS = \emptyset$ and every element of $\mathcal{S}$ is a loop in $M(\mathcal{S})$, we obtain $\supp(A) \cap \mathcal{S} = \emptyset$, which yields $\supp(A) = [n] \setminus \mathcal{S}$.

\smallskip

For the converse, assume that conditions (i) and (ii) hold for $A$, and we aim to show that $A \in \TS$. That is, for any $\mathcal{S}' \in \min(C_n) \setminus \{\mathcal{S}\}$, we must verify that equation~\eqref{nonempty} holds. We distinguish two cases:

\smallskip
\noindent\textbf{Case 1:} $\mathcal{S}' \not\subset \mathcal{S}$. Choose $x \in \mathcal{S}' \setminus \mathcal{S}$. Then $x \in [n] \setminus \mathcal{S} = \supp(A)$, so there exists $e \in A$ with $x \in e$. Since $x \in \mathcal{S}'$, it is a loop in $M(\mathcal{S}')$, and thus $e \in \mathcal{D}(M(\mathcal{S}'))$. Since $A \cap \DMS=\emptyset$, the intersection in~\eqref{nonempty} is nonempty.

\smallskip
\noindent\textbf{Case 2:} $\mathcal{S}' \subsetneq \mathcal{S}$. Then there exists $j \in [k]$ such that $b_j + 1 \in \mathcal{S}\setminus \mathcal{S}'$. Hence, $I_j \cup \{b_j + 1\} \cup I_{j+1} \subseteq [n] \setminus \mathcal{S}'$. In particular, any pair $\{u,v\}$ with $u \in I_j$ and $v \in I_{j+1}$ is dependent in $M(\mathcal{S}')$. By condition (ii), $A$ contains such a pair, and so the intersection in~\eqref{nonempty} is again nonempty.

\smallskip

This proves that $A \in \TS$, completing the proof.
\end{proof}

From the description of $\mathcal{T}(\mathcal{S})$, one derives a finite generating set for the ideal $\JS$, as given in Definition~\ref{notation JS}. Our goal now is to provide a smaller generating set. We first introduce some notation.

\begin{notation}\label{F and P}
Let $F := \cup_{j=1}^{k} \{a_j, b_j\}$. We define the polynomial
\[
P := \prod_{j=1}^{k} f_{b_j, a_{j+1}}= \prod_{j=1}^{k} (x_{b_{j}}y_{a_{j+1}}-x_{a_{j+1}}y_{b_{j}}).
\]
\end{notation}

Before stating the next lemma, recall the definition of $\JS$ from Definition~\ref{notation JS} and the definition of $g_{C,D}$ from Definition~\ref{important notation}.

\begin{lemma}\label{concise generating set 2}
Using Notation~\ref{F and P}, we have the following equality of ideals:
\begin{equation}\label{equality JS2}
\sqrt{J_{C_{n}} + \JS} = \sqrt{J_{C_{n}} + \left( P \cdot g_{C,D} \ : \ C \amalg D = [n] \setminus (\mathcal{S} \cup F) \right)}.
\end{equation}
\end{lemma}

\begin{proof}
Define the ideal
\[
J := \left( P \cdot g_{C,D} \ : \ C \amalg D = [n] \setminus (\mathcal{S} \cup F) \right).
\]
To prove the equality in~\eqref{equality JS2}, it suffices to show the equivalent identity of varieties:
\begin{equation}\label{equality of varieties}
\mathbb{V}(J_{C_{n}}) \cap \mathbb{V}(\JS) = \mathbb{V}(J_{C_{n}}) \cap \mathbb{V}(J).
\end{equation}
The inclusion ``$\subset$'' in~\eqref{equality of varieties} is immediate from the fact that the generators of~$J$ form a subset of those of~$\JS$.

For the reverse inclusion, let $\gamma = (\gamma_1, \ldots, \gamma_n) \in (\CC^2)^n$ be a point in the right-hand side of~\eqref{equality of varieties}. We aim to show that $\gamma \in \mathbb{V}(\JS)$, i.e., that $g_{A,C,D}(\gamma) = 0$ for every generator $g_{A,C,D}$ of $\JS$. We consider two cases.

\smallskip
\noindent
\textbf{Case 1:} $\gamma_i = 0$ for some $i \in [n] \setminus \mathcal{S}$. Since $A \in \TS$, Lemma~\ref{charact TS} implies that $\supp(A) = [n] \setminus \mathcal{S}$, so there exists $e \in A$ such that $i \in e$.

If $e = \{i\}$, then either $x_i$ or $y_i$ divides $g_{C,D}$, implying $g_{C,D}(\gamma) = 0$, and hence $g_{A,C,D}(\gamma) = 0$.

If $e = \{i,j\}$, then $f_{i,j}$ divides $g_{A,C,D}$. Since $\gamma_i = 0$, it follows that $g_{A,C,D}(\gamma) = 0$.

\smallskip
\noindent
\textbf{Case 2:} $\gamma_i \neq 0$ for all $i \in [n] \setminus \mathcal{S}$. In particular, for every $i \in [n] \setminus (\mathcal{S} \cup F)$, at least one coordinate of $\gamma_i$ is nonzero. Therefore, there exists a partition $\{C', D'\}$ of $[n] \setminus (\mathcal{S} \cup F)$ such that $g_{C',D'}(\gamma) \neq 0$.

Since $\gamma \in \mathbb{V}(J)$, we have
\[
g_{C',D'}(\gamma) \cdot P(\gamma) = 0,
\]
and thus $P(\gamma) = 0$. Hence, for some $j \in [k]$, we must have $f_{b_j,a_{j+1}}(\gamma) = 0$, meaning that $\gamma_{b_j}$ and $\gamma_{a_{j+1}}$ are linearly dependent.

Moreover, since $A \in \TS$, there exist $u \in I_j$ and $v \in I_{j+1}$ such that $\{u,v\} \in A$. As $\gamma \in \mathbb{V}(J_{C_{n}})$, all the pairs
\begin{equation}\label{linearly dependent}
\{\gamma_u,\gamma_{u+1}\}, \ldots, \{\gamma_{b_j - 1}, \gamma_{b_j}\} \quad \text{and} \quad \{\gamma_{a_{j+1}}, \gamma_{a_{j+1}+1}\}, \ldots, \{\gamma_{v-1}, \gamma_v\}
\end{equation}
are linearly dependent. Combined with the dependence of $\gamma_{b_j}$ and $\gamma_{a_{j+1}}$, and using that none of the vectors in~\eqref{linearly dependent} is zero, it follows that $\gamma_u$ and $\gamma_v$ are linearly dependent.

Since $\{u,v\} \in A$, we conclude that $f_{u,v}$ divides $g_{A,C,D}$, and therefore $g_{A,C,D}(\gamma) = 0$. This completes the proof.
\end{proof}

\begin{definition}\label{simplified notation}
For ease of notation, let $I_{\mathcal{S}}$ denote the ideal
\[
I_{\mathcal{S}} := \left( P \cdot g_{C,D} \ : \ C \amalg D = [n] \setminus (\mathcal{S} \cup F) \right),
\]
and let $\mathcal{G}_{\mathcal{S}}$ be its generating set.
\end{definition}

The next natural step, following the strategy outlined in Subsection~\ref{reduction with transversals}, is to determine the radical of $J_{C_n} + \JS$. While this would typically require computing a Gröbner basis for the ideal, in light of Lemma~\ref{concise generating set 2}, we instead focus on computing a Gröbner basis for the ideal $J_{C_n} + I_{\mathcal{S}}$. To carry out this computation, we must first find an appropriate monomial ordering. The monomial order introduced in the following definition is a relabeling of the lexicographical order from Definition~\ref{lexicographical order}, yielding a modified version adapted to the new vertex labeling.

\begin{definition}\label{sigma admissible}
Let $\sigma$ be a permutation of $[n]$. We define $<_{\sigma}$ to be the lexicographic term order on the polynomial ring $\CC[x_1,\ldots,x_n, y_1,\ldots,y_n]$ induced by
\[
x_{\sigma^{-1}(1)} > x_{\sigma^{-1}(2)} > \cdots > x_{\sigma^{-1}(n)} > y_{\sigma^{-1}(1)} > y_{\sigma^{-1}(2)} > \cdots > y_{\sigma^{-1}(n)}.
\]
Let $i,j$ be vertices with $\sigma(i) < \sigma(j)$. A path 
    \[
    \pi: i = i_0, i_1, \ldots, i_r = j
    \]
    from $i$ to $j$ is called $\sigma$-\emph{admissible} if the following conditions hold:
    \begin{itemize}
        \item[{\rm (i)}] The vertices $i_0, \ldots, i_r$ are pairwise distinct.
        \item[{\rm (ii)}] For each $1 \leq k \leq r-1$, either $\sigma(i_k) < \sigma(i)$ or $\sigma(i_k) > \sigma(j)$.
        \item[{\rm (iii)}] For any proper subset $\{j_1, \ldots, j_s\}$ of $\{i_1, \ldots, i_{r-1}\}$, the sequence $i, j_1, \ldots, j_s, j$ is not a path in $G$.
    \end{itemize}
    To a $\sigma$-admissible path $\pi$, we associate the monomial
    \[
    u_\pi = \prod_{\sigma(i_k) > \sigma(j)} x_{i_k} \cdot \prod_{\sigma(i_l) < \sigma(i)} y_{i_l}.
    \]
\end{definition}

\begin{definition} \label{condition v}
Let $\sigma$ be a permutation of $[n]$. We say that $\sigma$ is \emph{$\mathcal{S}$-consistent} if it satisfies the following conditions for each $j \in [k]$:
\begin{enumerate}
    \item[{\rm (i)}] If $\sigma(b_j) < \sigma(a_{j+1})$ and $|I_{j+1}| \geq 2$, then $\sigma(a_{j+1}) < \sigma(a_{j+1}+1)$.
    \item[{\rm (ii)}] If $\sigma(b_j) < \sigma(a_{j+1})$ and $|I_j| \geq 2$, then $\sigma(b_j - 1) < \sigma(b_j)$.
    \item[{\rm (iii)}] If $\sigma(b_j) > \sigma(a_{j+1})$ and $|I_{j+1}| \geq 2$, then $\sigma(a_{j+1}) > \sigma(a_{j+1}+1)$.
    \item[{\rm (iv)}] If $\sigma(b_j) > \sigma(a_{j+1})$ and $|I_j| \geq 2$, then $\sigma(b_j - 1) > \sigma(b_j)$.
    \item[{\rm (v)}] If $\size{I_{j}}\geq 3$, we have
    \[\sigma(a_{j})<\cdots <\sigma(b_{j}) \quad \text{or} \quad \sigma(a_{j})>\cdots >\sigma(b_{j}).\]
\end{enumerate}
\end{definition}

We now describe a Gröbner basis for the ideal $J_{C_{n}} + I_{\mathcal{S}}$ under the assumption that there exists a permutation $\sigma$ that is $\mathcal{S}$-consistent.

\begin{proposition}\label{grobner basis for Cn}
Let $\sigma$ be an $\mathcal{S}$-consistent permutation. Using the notation from Definitions~\ref{simplified notation} and~\ref{sigma admissible}, the set
\[
\mathcal{G} = \bigcup_{\sigma(i)<\sigma(j)} \left\{ u_{\pi} f_{i,j} \,\middle|\, \pi \text{ is a $\sigma$-admissible path from } i \text{ to } j \right\} \cup \GS
\]
is a Gröbner basis for the ideal $J_{C_{n}} + I_{\mathcal{S}}$ with respect to the lexicographic order $<_{\sigma}$.
\end{proposition}

\begin{proof}
Set 
\[
\mathcal{G}' = \bigcup_{\sigma(i) < \sigma(j)} \left\{ u_{\pi} f_{i,j} \,\middle|\, \pi \text{ is a $\sigma$-admissible path from } i \text{ to } j \right\}.
\]
By Proposition~\ref{grobner basis for JG}~(i), the union $\mathcal{G}' \cup \GS$ generates the ideal $J_{C_{n}} + I_{\mathcal{S}}$. To show that $\mathcal{G}$ is a Gröbner basis, we use Buchberger’s criterion and verify that every $S$-pair $S(f,h)$ with $f,h \in \mathcal{G}$ reduces to zero. We proceed by analyzing the different possible cases for the pair $(f,h)$.

\smallskip
\noindent\textbf{Case~1:} $f,h \in \mathcal{G}'$. Under the relabeling $(1,\ldots,n) \mapsto (\sigma(1),\ldots,\sigma(n))$, Proposition~\ref{grobner basis for JG}~(i) implies that $\mathcal{G}'$ is a Gröbner basis for $J_{C_{n}}$. Consequently, the $S$-polynomial $S(f,h)$ reduces to zero.

\smallskip
\noindent\textbf{Case~2:} $f,h \in \GS$. In this situation, one has $S(f,h)=0$.

\smallskip
\noindent\textbf{Case~3:} $f\in \GS$ and $h\in \mathcal{G}'$. Let $f=P\cdot g_{C,D}$ and $h=u_{\pi}f_{i,l}$ where $\pi$ is a $\sigma$-admissible path from $i$ to $l$. We distinguish two subcases:

\smallskip  
\noindent\textbf{Case~3.1:} $\pi \cap \mathcal{S} \neq \emptyset$. In this case, there exists an index $j$ such that $b_j + 1 \in \pi \cap \mathcal{S}$. The corresponding $S$-polynomial is given by  
\begin{equation} \label{equation S expression}
S(f,h) = \frac{\mathrm{lcm}(\mathrm{in}_{<_{\sigma}}(f),\mathrm{in}_{<_{\sigma}}(h))}{\mathrm{in}_{<_{\sigma}}(f)} \cdot f - \frac{\mathrm{lcm}(\mathrm{in}_{<_{\sigma}}(f),\mathrm{in}_{<_{\sigma}}(h))}{\mathrm{in}_{<_{\sigma}}(h)} \cdot h.
\end{equation}  
Since $b_j + 1 \in \pi$, it follows that either $x_{b_j+1} \mid \mathrm{in}_{<_{\sigma}}(h)$ or $y_{b_j+1} \mid \mathrm{in}_{<_{\sigma}}(h)$. Moreover, because $b_j + 1 \in \mathcal{S}$, we know that neither $x_{b_j+1}$ nor $y_{b_j+1}$ divides $\mathrm{in}_{<_{\sigma}}(f)$. Consequently,  
\[
x_{b_j+1} \mid \frac{\mathrm{lcm}(\mathrm{in}_{<_{\sigma}}(f),\mathrm{in}_{<_{\sigma}}(h))}{\mathrm{in}_{<_{\sigma}}(f)} \quad \text{or} \quad y_{b_j+1} \mid \frac{\mathrm{lcm}(\mathrm{in}_{<_{\sigma}}(f),\mathrm{in}_{<_{\sigma}}(h))}{\mathrm{in}_{<_{\sigma}}(f)}.
\]  
Since $f_{b_j, a_{j+1}} \mid f$, we obtain  
\begin{equation} \label{in Jg}
f_{b_j, a_{j+1}} x_{b_j+1} \mid \frac{\mathrm{lcm}(\mathrm{in}_{<_{\sigma}}(f),\mathrm{in}_{<_{\sigma}}(h))}{\mathrm{in}_{<_{\sigma}}(f)} \cdot f \quad \text{or} \quad f_{b_j, a_{j+1}} y_{b_j+1} \mid \frac{\mathrm{lcm}(\mathrm{in}_{<_{\sigma}}(f),\mathrm{in}_{<_{\sigma}}(h))}{\mathrm{in}_{<_{\sigma}}(f)} \cdot f.
\end{equation}  
Since $b_j, b_j+1, a_{j+1}$ form a path in $C_n$, it follows from Proposition~\ref{grobner basis for JG}~(ii) that both $f_{b_j, a_{j+1}} x_{b_j+1}$ and $f_{b_j, a_{j+1}} y_{b_j+1}$ belong to $J_{C_n}$. Thus, by \eqref{in Jg}, we conclude that  
\[
\frac{\mathrm{lcm}(\mathrm{in}_{<_{\sigma}}(f),\mathrm{in}_{<_{\sigma}}(h))}{\mathrm{in}_{<_{\sigma}}(f)} \cdot f \in J_{C_n}.
\]  
Moreover, since $h \in J_{C_n}$, it follows from \eqref{equation S expression} that $S(f,h) \in J_{C_n}$. Finally, since $\mathcal{G}'$ is a Gröbner basis for $J_{C_{n}}$ by Proposition~\ref{grobner basis for JG}~(i), we conclude that $S(f,h)$ reduces to zero.

\smallskip  
\noindent\textbf{Case~3.2:} $\pi \cap \mathcal{S}= \emptyset$. We distinguish two subcases:

\smallskip  
\noindent\textbf{Case~3.2.1:} $\{i, l\} \cap F = \emptyset$. In this case, we have $\mathrm{in}_{<_{\sigma}}(f_{i,l}) = x_i y_l$, while the support of $\mathrm{in}_{<_{\sigma}}(P)$ is contained in $F$. It follows that $\mathrm{in}_{<_{\sigma}}(f_{i,l})$ and $\mathrm{in}_{<_{\sigma}}(P)$ form a regular sequence. Therefore, by reasoning as in Case~2.2 of the proof of Proposition~\ref{proposition grobner basis JG+JS}, we conclude that $S(f,h)$ reduces to zero.

\smallskip  
\noindent\textbf{Case~3.2.2:} $\{i, l\} \cap F \neq \emptyset$. Suppose first that $i$ and $l$ belong to distinct intervals $I_{j_1}$ and $I_{j_2}$. Then any path $\pi$ connecting them must traverse an element of $\mathcal{S}$, contradicting the assumption that $\pi \cap \mathcal{S} = \emptyset$. Thus, there exists an index $j$ such that $\{i, l\} \subset I_j$, and the path $\pi$ lies entirely within $I_j$.

Since $\sigma$ is $\mathcal{S}$-consistent, Condition~(v) from Definition~\ref{condition v} ensures that
\[
\sigma(i) < \sigma(x) < \sigma(l)
\quad \text{for all } x \in \pi \setminus \{i, l\}.
\]
This contradicts the assumption that $\pi$ is $\sigma$-admissible unless $\pi = \{i, l\}$. Hence, $\{i, l\} \in E(C_n)$.

\smallskip  
\noindent\textbf{Case~3.2.2.1:} $|I_j| = 2$, that is, $b_j = a_j + 1$. In this case, we must have $(i, l) = (a_j, b_j)$ or $(i, l) = (b_j, a_j)$. As the argument is symmetric, we may assume without loss of generality that $i = a_j$ and $l = b_j$.

Since $\sigma(a_j)=\sigma(b_{j}-1) < \sigma(b_j)=\sigma(a_{j}+1)$ and $\sigma$ is $\mathcal{S}$-consistent, Conditions~(iii) and (iv) from Definition~\ref{condition v} imply
\[
\sigma(b_{j-1}) < \sigma(a_j) 
\quad \text{and} \quad 
\sigma(b_j) < \sigma(a_{j+1}).
\]
It follows that
\[
\mathrm{in}_{<_{\sigma}}(f_{i,l}) = x_{a_j} y_{b_j}, 
\quad 
\mathrm{in}_{<_{\sigma}}(f_{b_{j-1}, a_j}) = x_{b_{j-1}} y_{a_j}, 
\quad \text{and} \quad 
\mathrm{in}_{<_{\sigma}}(f_{b_j, a_{j+1}}) = x_{b_j} y_{a_{j+1}}.
\]
Thus, $\mathrm{in}_{<_{\sigma}}(f_{i,l})$ and $\mathrm{in}_{<_{\sigma}}(f_{b_{j-1}, a_j} \cdot f_{b_j, a_{j+1}})$ are coprime, so $\mathrm{in}_{<_{\sigma}}(f_{i,l})$ and $\mathrm{in}_{<_{\sigma}}(P)$ form a regular sequence. Consequently, by the same reasoning as in Case~2.2 of the proof of Proposition~\ref{proposition grobner basis JG+JS}, we conclude that $S(f, h)$ reduces to zero.

\smallskip  
\noindent\textbf{Case~3.2.2.2:} $|I_j| \geq 3$. Since $\{i, l\} \cap F \neq \emptyset$, the pair $(i, l)$ must be one of $(a_j, a_j + 1)$, $(a_{j}+1, a_j)$, $(b_j - 1, b_j)$, or $(b_j, b_j - 1)$. By symmetry, it suffices to consider the case $i = a_j$ and $l = a_j + 1$.

Given that $\sigma(a_j) < \sigma(a_j + 1)$ and $\sigma$ is $\mathcal{S}$-consistent, Condition~(i) from Definition~\ref{condition v} implies
\[
\sigma(b_{j-1}) < \sigma(a_j).
\]
Hence,
\[
\mathrm{in}_{<_{\sigma}}(f_{i,l}) = x_{a_j} y_{a_j + 1} 
\quad \text{and} \quad 
\mathrm{in}_{<_{\sigma}}(f_{b_{j-1}, a_j}) = x_{b_{j-1}} y_{a_j}.
\]
These two initial terms are coprime, and thus $\mathrm{in}_{<_{\sigma}}(f_{i,l})$ and $\mathrm{in}_{<_{\sigma}}(P)$ form a regular sequence. Therefore, by the same argument as in Case~2.2 of the proof of Proposition~\ref{proposition grobner basis JG+JS}, it follows that $S(f, h)$ reduces to zero. This completes the proof.
\end{proof}

We now present a sufficient condition on $\mathcal{S}$ that ensures the existence of an $\mathcal{S}$-consistent permutation. Recall the notions of $C_{1}(\mathcal{S})$ and $C_{2}(\mathcal{S})$ from Notation~\ref{notation c1 and c2}.

\begin{lemma}\label{two of size one}
If $\size{C_{1}(\mathcal{S})}\geq 2$, then there exists an $\mathcal{S}$-consistent permutation of $[n]$.
\end{lemma}

\begin{proof}
Let $j_1$ and $j_2$ be distinct indices such that $|I_{j_1}| = |I_{j_2}| = 1$, and let $I_{j_1} = \{r\}$ and $I_{j_2} = \{s\}$. Define a permutation $\sigma$ of $[n]$ such that:
\begin{itemize}
\item[{\rm (i)}] $\sigma(r) = n$ and $\sigma(s) = 1$,
\item[{\rm (ii)}] $\sigma(s) < \sigma(s+1) < \cdots < \sigma(r-1) < \sigma(r)$ and $\sigma(s) < \sigma(s-1) < \cdots < \sigma(r+1) < \sigma(r)$.
\end{itemize}
It is straightforward to verify that $\sigma$ satisfies conditions~(i)--(v) from Definition~\ref{condition v}.
\end{proof}

We now address Question~\ref{localized for cycles} under the assumption that the complement of $\mathcal{S}$ contains at least two intervals of size one. One key tool is the next proposition from \cite{dey2024v}. Recall that we denote $S=\CC[x_{1},\ldots,x_{n},y_{1},\ldots,y_{n}]$ and $<$ denotes the monomial order from Definition~\ref{lexicographical order}.

\begin{proposition}\textup{\cite[Proposition~3.3]{dey2024v}}\label{proposition dey}
Let $G$ be a graph and let $f \in S_{m}$ be a homogeneous polynomial of degree $m$ such that $(J_{G} : f) = P_{\mathcal{S}}$ for some minimal prime $P_{\mathcal{S}}$ of $J_{G}$. Then, there exists an element $g \in S_{m}$ such that $\mathrm{in}_{<}(g) \notin \mathrm{in}_{<}(J_{G})$ and $(J_{G} : g) = P_{\mathcal{S}}$.
\end{proposition}

The next proposition provides effective upper and lower bounds for the invariant $\VCN$ when $\size{C_{1}(\mathcal{S})} \geq 2$.

\begin{proposition}\label{bound interval}
Suppose $\size{C_{1}(\mathcal{S})}\geq 2$. Then
\begin{equation}\label{inequality}
n-\size{C_{2}(\mathcal{S})}-2\leq \VCN \leq n - \left|C_{2}(\mathcal{S})\right|.
\end{equation}
\end{proposition}

\begin{proof}
The right-hand inequality in~\eqref{inequality} follows from~\textup{\cite[Theorem~4.7]{dey2024v}}. We now establish the left-hand inequality.

Let $j_1$ and $j_2$ be distinct indices such that $|I_{j_1}| = |I_{j_2}| = 1$, and let $I_{j_1} = \{r\}$ and $I_{j_2} = \{s\}$. Let $\sigma$ be the permutation constructed in the proof of Lemma~\ref{two of size one}. By that lemma, $\sigma$ is $\mathcal{S}$-consistent. Then, by Proposition~\ref{grobner basis for Cn}, a Gröbner basis for the ideal $J_{C_n} + I_{\mathcal{S}}$ with respect to the lexicographic order $<_{\sigma}$ is given by
\[
\bigcup_{\sigma(i) < \sigma(j)} \left\{ u_\pi f_{i,j} \,\middle|\, \pi \text{ is a $\sigma$-admissible path from } i \text{ to } j \right\} \cup \GS.
\]
It follows that
\begin{equation}\label{initial ideal}
\mathrm{in}_{<_{\sigma}}(J_{C_n} + I_{\mathcal{S}}) = \mathrm{in}_{<_{\sigma}}(J_{C_n}) + \left(\mathrm{in}_{<_{\sigma}}(P)\cdot g_{C,D} \,\middle|\, C \amalg D = [n] \setminus (\mathcal{S} \cup F)\right).
\end{equation}

Moreover, we compute:
\begin{equation}\label{in P}
\begin{aligned}
\mathrm{in}_{<_{\sigma}}(P)
&= \prod_{j} \mathrm{in}_{<_{\sigma}}(f_{b_j,a_{j+1}}) \\
&= y_r^2 x_s^2 \prod_{j \in [j_{1}+1, j_{2}-2]} \mathrm{in}_{<_{\sigma}}(f_{b_j,a_{j+1}}) \prod_{j \in [j_{2}+1, j_{1}-2]} \mathrm{in}_{<_{\sigma}}(f_{b_j,a_{j+1}}) \\
&= y_r^2 x_s^2 \prod_{j \in [j_{1}+1, j_{2}-2]} x_{a_{j+1}} y_{b_j} \prod_{j \in [j_{2}+1, j_{1}-2]} x_{b_j} y_{a_{j+1}}.
\end{aligned}
\end{equation}

Combining~\eqref{initial ideal} and~\eqref{in P}, we obtain:
\begin{equation}\label{radical initial}
\sqrt{\mathrm{in}_{<_{\sigma}}(J_{C_n} + I_{\mathcal{S}})} = \mathrm{in}_{<_{\sigma}}(J_{C_n}) + \left(U \cdot g_{C,D} \,\middle|\, C \amalg D = [n] \setminus (\mathcal{S} \cup F)\right),
\end{equation}
where
\[
U = y_r x_s \prod_{j \in [j_{1}+1, j_{2}-2]} x_{a_{j+1}} y_{b_j} \prod_{j \in [j_{2}+1, j_{1}-2]} x_{b_j} y_{a_{j+1}}.
\]

Moreover, note that $\deg(U)=\deg(P)-2=2\size{\mathcal{S}}-2$. Assume that $(J_{C_n} : f) = P_{\mathcal{S}}$ for some homogeneous polynomial $f$ of degree $d$. We aim to show that $d \geq n - \left|C_2(\mathcal{S})\right|-2$.

By Proposition~\ref{proposition dey}, we may assume that $\mathrm{in}_{<_{\sigma}}(f) \notin \mathrm{in}_{<_{\sigma}}(J_{C_n})$. Also note that
\[(J_{C_n} : P_{\mathcal{S}})=\sqrt{J_{C_{n}}+\JS}=\sqrt{J_{C_{n}}+I_{\mathcal{S}}},\]
where the first equality follows from Theorem~\ref{theorem transversal} and the second equality follows from Lemma~\ref{concise generating set 2}. Since $f \in (J_{C_n} : P_{\mathcal{S}})$, we have
\[
\mathrm{in}_{<_{\sigma}}(f) \in \mathrm{in}_{<_{\sigma}}(\sqrt{J_{C_n} + I_{\mathcal{S}}}) \subset \sqrt{\mathrm{in}_{<_{\sigma}}(J_{C_n} + I_{\mathcal{S}})}.
\]
Using~\eqref{radical initial} and the fact that $\mathrm{in}_{<_{\sigma}}(f) \notin \mathrm{in}_{<_{\sigma}}(J_{C_n})$, we conclude that
\[
\mathrm{in}_{<_{\sigma}}(f) \in \left(U \cdot g_{C,D} \,\middle|\, C \amalg D = [n] \setminus (\mathcal{S} \cup F)\right).
\]

Each generator of this monomial ideal has degree:
\begin{align*}
\deg(U) + |[n] \setminus (\mathcal{S} \cup F)|
&= \deg(P)-2 + |[n] \setminus (\mathcal{S} \cup F)| \\
&= 2|\mathcal{S}| - 2 + n - |F| - |\mathcal{S}| \\
&= n - |C_2(\mathcal{S})| - 2.
\end{align*}
Therefore,
\[
\deg(f) = \deg(\mathrm{in}_{<_{\sigma}}(f)) \geq n - |C_2(\mathcal{S})| - 2,
\]
completing the proof.
\end{proof}

By the preceding proposition, 
it remains to consider the case in which the complement of $\mathcal{S}$ contains at most one interval of size one. To address this, we now develop a series of auxiliary results needed to resolve the problem in this remaining case. 

As suggested by the preceding results, the initial ideal will play a central role in our analysis. In this direction, our next goal is to understand the initial ideals of $(J_{C_n} : x_i)$ and $(J_{C_n} : f_{i-1,i+1})$. 

\begin{lemma}\label{saturation with xi}
The following inclusion of ideals holds:
\begin{equation}\label{inclusion with saturation}
\mathrm{in}_{<}\big((J_{C_n} : x_i)\big) \subset \mathrm{in}_{<}(J_{C_n}) + (x_{i-1}, y_{i-1}) \cdot (x_{i+1}, y_{i+1}).
\end{equation}
\end{lemma}

\begin{proof}
By \textup{\cite[Proposition~3.1]{ambhore2024v}}, we have
\[
(J_{C_n} : x_i) = J_{C_n^i},
\]
where $C_n^i$ is the graph with vertex set $V(C_n^i) = [n]$ and edge set
\[
E(C_n^i) = E(C_n) \cup \{\{i-1, i+1\}\}.
\]
Moreover, Proposition~\ref{grobner basis for JG}(i) ensures that
\[
\mathcal{G}_1 = \bigcup_{j<l} \left\{ u_{\pi} f_{j,l} \,\middle|\, \pi \text{ is an admissible path in } C_n \text{ from } j \text{ to } l \right\}
\]
is a Gröbner basis for $J_{C_n}$, while
\[
\mathcal{G}_2 = \bigcup_{j<l} \left\{ u_{\pi} f_{j,l} \,\middle|\, \pi \text{ is an admissible path in } C_n^i \text{ from } j \text{ to } l \right\}
\]
is a Gröbner basis for $J_{C_n^i}$.

To prove the inclusion in~\eqref{inclusion with saturation}, let $u \in \mathrm{in}_{<}(J_{C_n^i})$ be a monomial. Since $\mathcal{G}_2$ is a Gröbner basis, there exists $u_{\pi} f_{j,l} \in \mathcal{G}_2$ such that $\mathrm{in}_{<}(u_{\pi} f_{j,l})$ divides $u$, where $\pi$ is an admissible path in $C_n^i$ from $j$ to $l$. We distinguish two cases.

\smallskip
\noindent\textbf{Case 1:} $\{i-1, i+1\} \not\subset \pi$. Then, all consecutive vertices in $\pi$ are adjacent in $C_n$, so $\pi$ is also an admissible path in $C_n$. It follows that $u_{\pi} f_{j,l} \in \mathcal{G}_1$, hence $\mathrm{in}_{<}(u_{\pi} f_{j,l}) \in \mathrm{in}_{<}(J_{C_n})$, and so $u \in \mathrm{in}_{<}(J_{C_n})$.

\smallskip
\noindent\textbf{Case 2:} $\{i-1, i+1\} \subset \pi$. Then $\mathrm{in}_{<}(u_{\pi} f_{j,l})$ is divisible by one of the monomials $x_{i-1}x_{i+1},\, x_{i-1}y_{i+1},\, y_{i-1}x_{i+1},$ or $y_{i-1}y_{i+1}$. Hence, $u \in (x_{i-1}, y_{i-1}) \cdot (x_{i+1}, y_{i+1})$.

In either case, we conclude that
\[
u \in \mathrm{in}_{<}(J_{C_n}) + (x_{i-1}, y_{i-1}) \cdot (x_{i+1}, y_{i+1}),
\]
as desired.
\end{proof}

Our next objective is to determine the initial ideal of $(J_{C_n} : f_{i-1,i+1})$. 

\begin{lemma}\label{saturation with f}
The initial ideal of the saturation $(J_{C_n} : f_{i-1,i+1})$ satisfies
\[
\mathrm{in}_{<}((J_{C_n} : f_{i-1,i+1})) = \mathrm{in}_{<}(J_{C_n}) + (x_i, y_i) + \left( g_{C,D} \,\middle|\, C \cup D = [n] \setminus \{i-1, i, i+1\},\; C \cap D = \emptyset \right).
\]
\end{lemma}

\begin{proof}
By applying Theorem~3.7 and Lemma~3.8 from \cite{mohammadi2014hilbert} to our setting, we obtain the following decomposition:
\[
(J_{C_n} : f_{i-1,i+1}) =  J_{C_n} + (x_i, y_i) + \left( g_{C,D} \,\middle|\, C \cup D = [n] \setminus \{i-1, i, i+1\},\; C \cap D = \emptyset \right).
\]
Let
\[
\mathcal{G}_1 = \bigcup_{j<l} \left\{ u_{\pi} f_{j,l} \,\middle|\, \pi \text{ is an admissible path in } C_n \text{ from } j \text{ to } l \right\},
\]
which forms a Gröbner basis for $J_{C_n}$ by Proposition~\ref{grobner basis for JG}(i). Then, using precisely the same argument as in the proof of \textup{\cite[Lemma~3.5]{jaramillo2024connected}}, it follows that the set
\[
\mathcal{G} = \mathcal{G}_1 \cup \{x_i, y_i\} \cup \left\{ g_{C,D} \,\middle|\, C \cup D = [n] \setminus \{i-1, i, i+1\},\; C \cap D = \emptyset \right\}
\]
is a Gröbner basis for $(J_{C_n} : f_{i-1,i+1})$. The desired equality of initial ideals then follows directly from this description.
\end{proof}

Applying Lemmas~\ref{saturation with xi} and~\ref{saturation with f}, we are now ready to establish a new lower bound for the localized v-number $\VCN$.

\begin{proposition}\label{bound interval 2}
Let $\mathcal{S} \in \min(C_n)$ with $\size{\mathcal{S}} \geq 3$. Then, we have the inequality
\[
\VCN \geq n - \size{\mathcal{S}}.
\]
\end{proposition}

\begin{proof}
To establish the result, assume that $(J_{C_n} : f) = P_{\mathcal{S}}$ for some homogeneous polynomial $f$, and we aim to show that $\deg(f) \geq n - \size{\mathcal{S}}$.

By Proposition~\ref{proposition dey}, we may assume $\mathrm{in}_{<}(f) \notin \mathrm{in}_{<}(J_{C_n})$. Since $f \in (J_{C_n} : P_{\mathcal{S}})$, it follows that $f \in (J_{C_n} : x_i)$ for every $i \in \mathcal{S}$. Consequently, by Lemma~\ref{saturation with xi}, we obtain
\[
\mathrm{in}_{<}(f) \in \mathrm{in}_{<}(J_{C_n}) + (x_{i-1}, y_{i-1}) \cdot (x_{i+1}, y_{i+1}).
\]
Using that $\mathrm{in}_{<}(f) \notin \mathrm{in}_{<}(J_{C_n})$, it follows that for each $i \in \mathcal{S}$, at least one of the following monomials divides $\mathrm{in}_{<}(f)$:
\[
x_{i-1}x_{i+1}, \quad x_{i-1}y_{i+1}, \quad y_{i-1}x_{i+1}, \quad y_{i-1}y_{i+1}.
\]
Hence, for every $\ell \in F$, we have
\begin{equation}\label{x y}
x_{\ell} \mid \mathrm{in}_{<}(f) \quad \text{or} \quad y_{\ell} \mid \mathrm{in}_{<}(f).
\end{equation}

Now consider any vertex $i \in [n] \setminus (\mathcal{S} \cup F)$. Since $f \in (J_{C_n} : P_{\mathcal{S}})$ and $f_{i-1,i+1} \in P_{\mathcal{S}}$, it follows that $f \in (J_{C_n} : f_{i-1,i+1})$. Consequently, by Lemma~\ref{saturation with f}, we obtain
\[
\mathrm{in}_{<}(f) \in \mathrm{in}_{<}(J_{C_n}) + (x_i, y_i) + \left( g_{C,D} \,\middle|\, C \cup D = [n] \setminus \{i-1, i, i+1\},\; C \cap D = \emptyset \right).
\]
Using again that $\mathrm{in}_{<}(f) \notin \mathrm{in}_{<}(J_{C_n})$, there are three possibilities: either $g_{C,D} \mid \mathrm{in}_{<}(f)$ for some partition $\{C,D\}$ of $[n] \setminus \{i-1, i, i+1\}$, or $x_i \mid \mathrm{in}_{<}(f)$ or $y_i \mid \mathrm{in}_{<}(f)$. In the first case, we obtain
\[
\deg(f) = \deg(\mathrm{in}_{<}(f)) \geq \deg(g_{C,D}) = n - 3 \geq n - \size{\mathcal{S}},
\]
since we are assuming $\size{\mathcal{S}}\geq 3$, which proves the statement. It remains to consider the other cases, where
\[
x_i \mid \mathrm{in}_{<}(f) \quad \text{or} \quad y_i \mid \mathrm{in}_{<}(f)
\]
for every $i \in [n] \setminus (\mathcal{S} \cup F)$. Together with~\eqref{x y}, we then have
\[
x_i \mid \mathrm{in}_{<}(f) \quad \text{or} \quad y_i \mid \mathrm{in}_{<}(f)
\]
for all $i \in [n] \setminus \mathcal{S}$. Consequently,
\[
\deg(f) = \deg(\mathrm{in}_{<}(f)) \geq n - \size{\mathcal{S}},
\]
as desired.
\end{proof}

By combining Propositions~\ref{bound interval} and~\ref{bound interval 2}, we obtain effective upper and lower bounds for the localized v-numbers, thereby limiting their possible values to a small number of cases.

\begin{proposition}\label{proposition localized v for Cn}
Let $\mathcal{S} \in \min(C_n)$ with $\size{\mathcal{S}} \geq 3$. Then the following holds:
\begin{enumerate}
\item[{\rm (i)}]\label{3} If $C_1(\mathcal{S}) = \emptyset$, then
\[
\VCN = n - \size{\mathcal{S}}.
\]
\item[{\rm (ii)}] If $\size{C_1(\mathcal{S})} = 1$, then
\[
n - \size{\mathcal{S}} \leq \VCN \leq n - \size{\mathcal{S}} + 1.
\]
\item[{\rm (iii)}] If $\size{C_1(\mathcal{S})} \geq 2$, then
\[
n - \size{C_2(\mathcal{S})} - 2 \leq \VCN \leq n - \size{C_2(\mathcal{S})}.
\]
\end{enumerate}
\end{proposition}

\begin{proof}
The upper bounds in each case follow from \textup{\cite[Theorem~4.7]{dey2024v}}, which gives $\VCN\leq n - \size{C_2(\mathcal{S})}$, and the fact that $\size{\mathcal{S}}=\size{C_1(\mathcal{S})}+\size{C_2(\mathcal{S})}$. The lower bounds in \textup{(i)} and \textup{(ii)} follow from Proposition~\ref{bound interval 2}, while that in \textup{(iii)} is a consequence of Proposition~\ref{bound interval}.
\end{proof}

We are now ready to present the main result of this section, which gives effective upper and lower bounds for the v-number of $C_n$. While this does not yield an explicit formula for $\VCNN$, it reduces its possible value to only two cases.

\begin{theorem}\label{main theorem 2}
The \textup{v}-number of $C_n$ satisfies the following:
\begin{enumerate}
\item[{\rm (i)}] If $n = 3k$, then $\VCNN = 2k$.
\item[{\rm (ii)}] If $n = 3k + 2$, then $\VCNN \in [2k + 1,\, 2k + 2]$.
\item[{\rm (iii)}] If $n = 3k + 1$, then $\VCNN \in [2k ,\, 2k + 1]$.
\end{enumerate}
\end{theorem}

\begin{proof}
We may assume that $n \geq 6$, as the statement holds for $n \leq 5$ by \textup{\cite[Propositions~4.8 and~4.9]{dey2024v}}. Using \textup{\cite[Corollary~4.10]{dey2024v}}, we obtain the upper bound $\VCNN \leq \textstyle \left\lceil \frac{2n}{3} \right\rceil$. Thus, it remains to establish the lower bounds: namely, that $\VCNN \geq \textstyle\left\lceil \frac{2n}{3} \right\rceil$ when $n \equiv 0 \pmod{3}$, and that $\VCNN \geq\textstyle \left\lceil \frac{2n}{3} \right\rceil - 1$ when $n \equiv 1,2 \pmod{3}$. To this end, we show that for every $\mathcal{S} \in \min(C_n)$, the localized v-number satisfies
\[
\VCN \geq 2k, \quad \VCN \geq 2k+1, \quad \text{or} \quad \VCN \geq 2k,
\]
according to whether $n = 3k$, $n = 3k+2$, or $n = 3k+1$, respectively.

For proving this, we will make repeated use of the following inequality, valid for any $\mathcal{S} \in \min(C_n)$:
\begin{equation}\label{equation with S}
n \geq \size{\mathcal{S}} + \size{C_{1}(\mathcal{S})} + 2\size{C_{2}(\mathcal{S})} = 3\size{\mathcal{S}} - \size{C_{1}(\mathcal{S})} = 2\size{C_{1}(\mathcal{S})} + 3\size{C_{2}(\mathcal{S})}.
\end{equation}

Moreover, if $\mathcal{S} = \emptyset$, it follows from \textup{\cite[Theorem~3.2]{jaramillo2024connected}} that $\VCN = n - 2$, and when $\abs{\mathcal{S}} = 2$, Theorem~\ref{main theorem} again gives $\VCN = n - 2$. Since we are assuming $n \geq 6$, we have $n - 2 \geq \left\lceil \tfrac{2n}{3} \right\rceil$. Thus, it suffices to consider the case where $\mathcal{S} \in \min(C_n)$ and $\abs{\mathcal{S}} \geq 3$.

\smallskip

\noindent
(i) When $n=3k$, we will show that $\VCN \geq 2k$ for every $\mathcal{S} \in \min(C_n)$ with $\abs{\mathcal{S}} \geq 3$, proceeding by cases:

\smallskip
\emph{Case 1: $\size{C_{1}(\mathcal{S})} \leq 1$.} Applying~\eqref{equation with S}, we obtain
\[
3k \geq 3\size{\mathcal{S}} - \size{C_{1}(\mathcal{S})} \geq 3\size{\mathcal{S}} - 1,
\]
which implies $\size{\mathcal{S}} \leq k$. Then, by Proposition~\ref{proposition localized v for Cn}~(i), (ii),
\[
\VCN \geq 3k - \size{\mathcal{S}} \geq 2k.
\]

\smallskip
\emph{Case 2: $\size{C_{1}(\mathcal{S})} \geq 2$.} Then~\eqref{equation with S} gives
\[
3k \geq 2\size{C_{1}(\mathcal{S})} + 3\size{C_{2}(\mathcal{S})} \geq 4 + 3\size{C_{2}(\mathcal{S})},
\]
yielding $\size{C_{2}(\mathcal{S})} \leq k - 2$. By Proposition~\ref{proposition localized v for Cn}~(iii),
\[
\VCN \geq 3k - \size{C_{2}(\mathcal{S})} - 2 \geq 2k.
\]

\smallskip
\noindent
(ii) When $n = 3k + 2$, we will show that $\VCN \geq 2k + 1$ for every $\mathcal{S} \in \min(C_n)$ with $\abs{\mathcal{S}} \geq 3$. We proceed as above.

\smallskip
\emph{Case 1: $\size{C_{1}(\mathcal{S})} \leq 1$.} Then~\eqref{equation with S} gives
\[
3k + 2 \geq 3\size{\mathcal{S}} - \size{C_{1}(\mathcal{S})} \geq 3\size{\mathcal{S}} - 1,
\]
hence $\size{\mathcal{S}} \leq k + 1$, and by Proposition~\ref{proposition localized v for Cn}~(i), (ii),
\[
\VCN \geq 3k + 2 - \size{\mathcal{S}} \geq 2k + 1.
\]

\smallskip
\emph{Case 2: $\size{C_{1}(\mathcal{S})} \geq 2$.} Then
\[
3k + 2 \geq 2\size{C_{1}(\mathcal{S})} + 3\size{C_{2}(\mathcal{S})} \geq 4 + 3\size{C_{2}(\mathcal{S})},
\]
so $\size{C_{2}(\mathcal{S})} \leq k - 1$. Proposition~\ref{proposition localized v for Cn}~(iii) then gives
\[
\VCN \geq 3k + 2 - \size{C_{2}(\mathcal{S})} - 2 \geq 2k + 1.
\]

\smallskip
\noindent
(iii) When $n = 3k + 1$, we will prove that $\VCN \geq 2k$ for every $\mathcal{S} \in \min(C_n)$ with $\abs{\mathcal{S}} \geq 3$.

\smallskip
\emph{Case 1: $\size{C_{1}(\mathcal{S})} \leq 1$.} Then~\eqref{equation with S} yields
\[
3k + 1 \geq 3\size{\mathcal{S}} - \size{C_{1}(\mathcal{S})} \geq 3\size{\mathcal{S}} - 1,
\]
so $\size{\mathcal{S}} \leq k$, and Proposition~\ref{proposition localized v for Cn}~(i), (ii) gives
\[
\VCN \geq 3k + 1 - \size{\mathcal{S}} \geq 2k.
\]

\smallskip
\emph{Case 2: $\size{C_{1}(\mathcal{S})} \geq 2$.} Then
\[
3k + 1 \geq 2\size{C_{1}(\mathcal{S})} + 3\size{C_{2}(\mathcal{S})} \geq 4 + 3\size{C_{2}(\mathcal{S})},
\]
implying $\size{C_{2}(\mathcal{S})} \leq k - 1$. Proposition~\ref{proposition localized v for Cn}~(iii) then ensures
\[
\VCN \geq 3k + 1 - \size{C_{2}(\mathcal{S})} - 2 \geq 2k. \qedhere
\]
\end{proof}

\smallskip

\noindent{\bf Acknowledgement.}  
The author warmly thanks Delio Jaramillo-Velez for bringing this problem to his attention and for his insightful suggestions. The author is also deeply grateful to Fatemeh Mohammadi for her insightful advice, and numerous valuable comments and discussions. This work was supported by the PhD fellowship 1126125N and partially funded by the FWO grants G0F5921N (Odysseus), G023721N, and the KU Leuven grant iBOF/23/064.

\bibliographystyle{abbrv}
\bibliography{Citation}

\medskip
{\footnotesize\noindent {\bf Authors' addresses}
\medskip

\noindent{Emiliano Liwski, 
KU Leuven}\hfill {\tt  emiliano.liwski@kuleuven.be}

\end{document}